\documentclass[a4paper,12pt]{article}
\usepackage{subcaption}
\usepackage{color}
\usepackage{latexsym}
\usepackage{amsmath}
\usepackage{amssymb}
\usepackage{enumerate}
\usepackage{pict2e}
\usepackage{graphicx}
\usepackage{float}
\usepackage{theorem}
\newtheorem{theorem}{Theorem}[section]
\newtheorem{proposition}[theorem]{Proposition}
\newtheorem{lemma}[theorem]{Lemma}
\newtheorem{corollary}[theorem]{Corollary}

\theorembodyfont{\rmfamily}
\newtheorem{proof}{\textmd{\textit{Proof.}}}

\newtheorem{remark}[theorem]{Remark}
\newtheorem{example}[theorem]{Example}
\newtheorem{definition}[theorem]{Definition}

\newtheorem{acknowledgement}{\textmd{\textit{Acknowledgements.}}}

\makeatletter

\@addtoreset{equation}{section}
\makeatother

\newcommand{\qedd}{\hfill \Box}

\newcommand{\R}{\ensuremath{\mathbb{R}}}
\newcommand{\Sph}{\ensuremath{\mathbb{S}}}

\title{The cut locus of a Randers rotational 2-sphere of revolution
\footnote{
Mathematics Subject Classification (2010)\,:\,53C60, 53C22.}
}

\author{Rattanasak HAMA, Jaipong KASEMSUWAN, Sorin V. SABAU}

\date{}
\begin{document}

\maketitle

\begin{abstract}
In the present paper we study structure of the cut locus of a Randers rotational 2-sphere of revolution $(M,F=\alpha+\beta)$. We show that in the case when Gaussian curvature of the Randers surface is monotone along a meridian the cut locus of a point $q\in M$ is a point on a subarc of the opposite half bending meridian or of the antipodal parallel (Theorem \ref{thm_F_cut_locus}). More generally, when the Gaussian curvature is not monotone along the meridian, but the cut locus of a point $q$ on the equator is a subarc of the same equator, then the cut locus of any point $\widetilde{q}\in M$ different from poles is a subarc of the antipodal parallel (Theorem \ref{thm_F_cut_2}). Some examples are also given at the last section.

\end{abstract}

\section{Introduction}

\quad The study of the global behaviour of geodesics, conjugate points and cut locus is a fundamental problem in modern differential geometry. In the Riemannian case, an extensive literature is available (see \cite{AT}, \cite{SST}, \cite{ST}), but in the more general case of a Finsler manifold, the results are not so easily obtained. The main difficulty is that the dependence of the metric on the direction implies the non-symmetry of the distance function and the non-reversibility of the geodesics.

 Finsler manifolds $(M,F)$ generalize the Riemannian ones in the sense that they are defined by a norm $F:TM\to[0,\infty)$ with the properties
\begin{itemize}
\item[(i)] $F$ is positive and differentiable on $\widetilde{TM}:=TM\setminus\{0\}$;
\item[(ii)] $F$ is $1$-positive homogeneous, i.e. $F(x,\lambda y)=\lambda\cdot F(x,y)$ for any $\lambda>0$ and for all $(x,y)\in\widetilde{TM}$;
\item[(iii)] the Hessian matrix $g_{ij}(x,y):=\frac{1}{2}\frac{\partial^2F^2(x,y)}{\partial y^i\partial y^j}$, $i,j\in\{1,...,n\}$, is positive definite on $\widetilde{TM}.$
\end{itemize}

 Here $TM$ denotes the tangent bundle of an $n$-dimensional smooth manifold $M$ and $(x,y)$ the canonical coordinates on $TM$. The Finsler structure is called absolute homogeneous if the homogeneity condition (ii) is replaced by $F(x,\lambda y)=\vert \lambda\vert\cdot F(x,y)$ for any $\lambda\in\R$. 

A Finsler norm $F$ determines and it is determined by its indicatrix bundle $SM:=\cup_{x\in M}S_xM$, where $S_xM:=\{y\in T_xM:F(x,y)=1\}$.

Obviously, the simplest Finsler manifolds are the Riemannian cases, but this is the trivial case for us.

Less trivial examples are deformations of Riemannian metrics by linear forms $\beta=b_i(x)y^i$ defined on $TM$. This type of Finsler manifolds include Randers, Kropina and Matsumoto metrics \cite{Mat}.

A Finsler norm can be used for defining the integral length $\mathcal{L}_F$ of a $C^\infty$ curve $\gamma:[a,b]\to M$ by
\begin{equation*}
\mathcal{L}_F(\gamma\big\vert_{[a,b]})=\int_a^bF(\gamma(t),\dot{\gamma}(t))dt,
\end{equation*}
where $\dot{\gamma}(t)=\frac{d\gamma}{dt}$ is the tangent vector of $\gamma$. This definition easily extends to the integral length of any piecewise $C^\infty$ curve on $M.$

A smooth curve $\gamma$ on a Finsler manifold that minimizes the integral length $\mathcal{L}_F$ over the set of all piecewise $C^\infty$ curves with fixed end points is called an {\it F-geodesic}.

Any $F$-geodesic $\gamma$ emanating from a point $p$ in a compact Finslerian (or Riemannian) manifold is losing its global minimizing property of a point $q$ on $\gamma$. Such point is called a {\it $F$-cut point of $p$} along $\gamma$. The {\it $F$-cut locus} of a point $p\in M$ is the set of all cut points along all geodesics emanating from $p$ on a Finsler manifold. This is an important geometrical object related to the topology of the manifold and to the global geometrical properties of the Finsler manifold.

Even though in general the cut locus may have a very complicated structure, it is known that the $F$-cut locus $\mathcal{C}^F_p$ of a point $p$ on a Finsler surface is a local tree and that any two points on the same connected component of $\mathcal{C}^F_p$ can be joined by a rectifiable Jordan arc in $\mathcal{C}^F_p$ (see \cite{SaT} for details).

Based on this theoretical result, we have studied in \cite{HCS} the actual structure of the cut locus of a point on a Randers rotational surface of revolution homeomorphic to $\R^2$.

The main aim of the present paper is to explicitly determine the structure of the cut locus of a point of a $2$-sphere of revolution endowed with a Randers rotational metric.

Randers metrics are special Finsler metrics whose indicatrices are obtained by rigid translations of the Riemannian unit sphere. It was Shen \cite{S} who pointed out for the first time that Randers metrics give solutions to the classical Zermelo's navigation problem, namely:\\

{\it Find the paths of  shortest time travel between two points under the influence of a wind or a current when we travel by a boat capable of a certain maximum speed.}\\

Formally, if we consider the background landscape to be a Riemannian manifold $(M,h)$, endowed with a vector field $W$ on $M$, $\Vert W \Vert_h <1$, then the shortest time travel paths are precisely the geodesics of a Finsler metric of Randers type
\begin{equation*}
F(x,y)=\alpha(x,y)+\beta(x,y)=\frac{\sqrt{\lambda\cdot\Vert y\Vert^2_h+W^2_0}}{\lambda}-\frac{W_0}{\lambda}
\end{equation*}
uniquely induced by the navigation data $(h,W)$. Here $W=W^i\cdot\frac{\partial}{\partial x^i}$ is the velocity vector field of the wind, $\lambda=1-\Vert W\Vert^2_h$, $W_0=h(W,y)$.

The corresponding Riemannian metric $\alpha=\sqrt{a_{ij}(x)y^iy^j}$ and $1$-form $\beta=b_i(x)y^i$ are given by
\begin{equation*}
a_{ij}(x)=\frac{\lambda\cdot h_{ij}+W_iW_j}{\lambda^2}\quad \text{and}\quad b_i=-\frac{W_i}{\lambda},
\end{equation*}
where $W_i:=h_{ij}W^j$.

This Randers metric satisfies all three conditions in the definition of a Finsler metric provided $\Vert W\Vert_h<1$ (see \cite{BCS}, \cite{BR}, \cite{R}, \cite{S} for details).

Our main theorems on the structure of the $F$-cut locus of a surface of revolution endowed with a Randers rotational metric are the following.

\begin{theorem}\label{thm_F_cut_locus}
Let $(M,F)$ be a Randers rotational 2-sphere of revolution with navigation data $(h,W)$, where $W=\mu\cdot\frac{\partial}{\partial \theta}$ is the wind blowing along parallels, $\mu<\{\frac{1}{\max\{m(r)\}}:r\in [0,2a]\}$, with a pair of poles $p,q$, $d_h(p,q)=2a$ and satisfying 
\begin{itemize}
\item $M$ is symmetric with respect to $\{r=a\}$,
\item the flag curvature $\mathcal{K}$ is monotone along a meridian.
\end{itemize}
Then the $F$-cut locus $\mathcal{C}^F_x$ of a point $x\in M\setminus\{p,q\}$ with $\{\theta(x)=0\}$ is
\begin{enumerate}
\item The subarc of the opposite half bending meridian,
\begin{equation*}
\mathcal{C}^F_x=\varphi(d(x,\tau(t)),\tau(t)),\quad t\in[c,2a-c],
\end{equation*}
where $\varphi$ is the flow of the wind, when $\mathcal{K}$ is monotone non-increasing. 
\item The following subarc of the antipodal parallel $\{r=2a-r(x)\}$ to $x$:
\begin{equation*}
\mathcal{C}_x^F=r^{-1}(2a-r(x))\cap \theta^{-1}\{\mathcal{H}(m)+\psi(x),2\pi-(\mathcal{H}(m)-\psi(x))\}.
\end{equation*}
where $\psi(x)=\mu\cdot d_h(x,\hat{q}_0)$, $\hat{q}_0$ is the $h$-first conjugate point of $x$ with respect to $h$, $m:=m(r(x))$, when $\mathcal{K}$ is monotone non-decreasing.
\item 
A single point on the antipodal parallel $\mathcal{C}^F_x=(2a-r(x),\pi(1+\mu R))$, where $R$ is radius of sphere, when $\mathcal{K}=\frac{1}{R^2}$ is constant.
\item If the cut locus of $x\in M\setminus\{p,q\}$ is a single point, then $\mathcal{K}$ is constant.
\end{enumerate}
\end{theorem}

More generally, if the Gaussian curvature of $h$, or of $F$, is not monotone, the following characterization of the cut locus is possible.
\begin{theorem}\label{thm_F_cut_2}
Let $(M,F=\alpha+\beta)$ be the Randers rotational 2-sphere of revolution constructed from the navigation data $(h,W)$ of a 2-sphere of revolution $(M,h)$.

If the $F$-cut locus of a point $x$ on the equator $\{r=a\}$ is a subarc of the equator $\{r=a\}$, then the $F$-cut locus of any point $\widetilde{x}$ with $r(\widetilde{x})\in(0,2a)\setminus\{a\}$ is a subarc of the antipodal parallel $\{r=2a-r(\widetilde{x})\}$.
\end{theorem}

This is a generalization of Theorem 3.5 in \cite{BCST} to the Randers case.

Here it is the structure of our paper.

We start by recalling the geometry of a Riemannian $2$-sphere of revolution and the structure of its cut locus (Section \ref{Riem_2_sph}). This section is an excerpt from \cite{ST}.

By using the navigation data $(h,W)$, where $h$ is the induced Riemannian metric on the $2$-sphere of revolution $M$, and $W:=\mu\cdot\frac{\partial}{\partial \theta}$ a mild wind blowing along the parallels, we construct in Section \ref{Ran_rota} a Randers rotational metric $F=\alpha+\beta$ on the $2$-sphere of revolution $M$. We determine the $F$-geodesic equations in Proposition \ref{F_Geodesic} and extend the Clairaut relation to $F$-geodesics. The conjugate and cut points along $F$-geodesics are obtained by mapping the conjugate and cut points along $h$-geodesics by means of the flow $\varphi$,  Propositions \ref{Prop_F_con} and \ref{prop_F_cut}, respectively.

Moreover, we show here that the flag curvature of this Randers metric coincide with the Gaussian curvature of $h$ (Lemma \ref{F_curvature}). Even though some of these results were proved already in \cite{HCS}, for a surface of revolution homeomorphic to $\R^2$, we show here how they extend to a $2$-sphere of revolution.

Section \ref{Proof of main theorem} is where we prove Theorem \ref{thm_F_cut_locus} by using a certain number of lemmas. Finally, in Section \ref{proof_thm_F_cut_2}, we prove Theorem \ref{thm_F_cut_2} and give some examples  of Randers rotational metrics whose Gaussian curvature is not monotone (Subsection \ref{ex_F_cut}).

We show that the convexity of the second derivative of the $F$-half period function different from the convexity of the $h$-half period function.

In a forthcoming research we will study the convexity of injectivity domain and other related topics of Randers rotational surface of revolution.

\begin{acknowledgement}
We express our gratitude to Prof. H. Shimada for many useful discussion and to Prof. M. Tanaka for this important suggestion.

The first author is grateful to Prof. P. Chitsakul for many years of supervision.
\end{acknowledgement}
\section{The 2-sphere of revolution}

\subsection{The Riemannian 2-sphere of revolution}\label{Riem_2_sph}


\quad A compact Riemannian manifold $(M,h)$ homeomorphic to a 2-sphere is called a {\it $2$-sphere of revolution} if $M$ admits a point $p$, called {\it pole}, such that for any two points $q_1,q_2$ on $M$ with $d_h(p,q_1)=d_h(p,q_2)$, there exists an $h$-isometry $f$ on $M$ satisfying $f(q_1)=q_2$, and $f(p)=p$,
where $d_h(\cdot,\cdot)$ denoted the $h$-Riemannian distance function on $M$.


Let $(r,\theta)$ denote geodesic polar coordinates around a pole $p$ of $(M,h)$. The Riemannian metric $้$ can be expressed as $h=dr^2+m^2(r)d\theta^2$ on $M\setminus\{p,q\}$, where $q$ denotes the unique $h$-cut point of $p$ and
\begin{equation*}
m(r(x)):=\sqrt{h\left(\left(\frac{\partial}{\partial\theta}\right)_x,\left(\frac{\partial}{\partial\theta}\right)_x\right)},
\end{equation*}
for any point $x\in M\setminus \{p,q\}$ with coordinates $(r(x),\theta(x))$ (see \cite{ST}).

It is known  that each pole of a 2-sphere of revolution $M$ has a unique cut point (see \cite{ST}, Lemma 2.1.). A pole and its unique cut point are called {\it a pair of poles}. 

From now, for the rest of the paper, we fix a pair of poles $p$, $q$ and the geodesic polar coordinates $(r,\theta)$ around $p$. 

\begin{remark}\label{rem_2_sphere_properties}
We always assume about $(M,h)$ the following conditions (as in \cite{ST}): 
\begin{itemize}
\item[1.] $M$ is symmetric with respect to the equator, i.e. reflection fixing $\{r=a\}$, where $d_h(p,q)=2a$. In other words, we assume 
$$
m(r)=m(2a-r),\quad \forall r\in (0,2a).
$$
\item[2.] The Gaussian curvature $G(x)=-\frac{m''(r(x))}{m(r(x))}$ of $(M,h)$ is monotone along the meridian from pole to the equator.
\end{itemize}
\end{remark}

We observe that both functions $m(r)$ and $m(2a-r)$ are extensible to a $C^\infty$ odd function around $\{r=0\}$ and $m'(0)=1=-m'(2a)$. 

Any periodic $h$-geodesic passing through a pair of poles is called a {\it meridian}, i.e. we have $\gamma(t)=\gamma(t+4a)$, for any $t\in \R$, and $p=\gamma(0)$. 

Any curve $r=c\in (0,2a)$ is called a 
{\it parallel}. The parallel $\{r=a\}$ is called the {\it equator} of $(M,h)$.

\begin{remark}
For the sake of simplicity we will often make use in the following of the Riemannian universal covering of $(M\setminus\{p,q\},dr^2+m(r)^2d\theta^2)$, namely 
\begin{equation*}
(\widetilde{M},\widetilde{h}):=((0,2a)\times\R,d\widetilde{r}^2+m(\widetilde{r})^2d\widetilde{\theta}^2),
\end{equation*}
with the covering projection
$\Pi:\widetilde{M}\to M\setminus\{p,q\}$.
\end{remark}

Recall that the equations of an $h$-unit speed geodesic $\gamma(s):=(r(s),\theta(s))$ of $(M,h)$ are
\begin{equation}\label{eq 4}
\begin{cases}
 \frac{d^2r}{ds^2}-mm'\left(\frac{d\theta}{ds}\right)^2=0\\ 
 \frac{d^2\theta}{ds^2}+2\frac{m'}{m}\left(\frac{dr}{ds}\right)\left(\frac{d\theta}{ds}\right)=0,
\end{cases}
\end{equation}
where $s$ is the arclength parameter of $\gamma$ with the $h$-unit speed parametrization condition 
\begin{equation}\label{eq 1}
\left(\frac{dr}{ds}\right)^2+m^2\left(\frac{d\theta}{ds}\right)^2 =1.
\end{equation}
It follows that every profile curve, or {\it meridian}, is an $h$-geodesic,
and that a parallel 
$\{r=r_0\}$ is geodesic, $r_0$ is constant, if and only if $m'(r_0)=0$.

We observe that  (\ref{eq 4}) implies
\begin{equation}\label{h-prime integral}
\frac{d\theta(s)}{ds}m^2(r(s)) = \nu,\quad \text{where $\nu$ is constant},
\end{equation}
that is, the quantity $\frac{d\theta}{ds}m^2$ is conserved along the $h$-geodesics. 

\begin{lemma}[The Clairaut relation]
Let $\widetilde{\gamma}(s)=(\widetilde{r}(s),\widetilde{\theta}(s))$ be an $\widetilde{h}$-unit speed geodesic on $(\widetilde{M},\widetilde h)$. There exists a constant $\nu$ such that 
\begin{equation}\label{Clair1}
m^2(\widetilde{r}(s))\widetilde{\theta}'(s)=m(\widetilde{r}(s))\cos\phi(s)=\nu
\end{equation}
hold for any $s$, where $\phi(s)$ denotes the angle between tangent vector of $\widetilde\gamma(s)$ and $\frac{\partial}{\partial \widetilde{\theta}}|_{ \widetilde\gamma(s)}$ (see Figure \ref{Clairaut}). The constant $\nu$ is called the {\it Clairaut constant} of $\widetilde{\gamma}$.
\end{lemma}

\begin{figure}
\setlength{\unitlength}{0.5cm}
\centering
\begin{picture}(10,10)
\put(1,0.5){\vector(0,1){9}}
\put(0.5,1){\vector(1,0){9.5}}
\put(0.5,9){\vector(1,0){9.5}}
\multiput(0.5,5)(0.5,0){19}{\line(1,0){0.3}}
\put(10,5){\vector(1,0){0.1}}
\put(5,0.5){\line(0,1){9}}
\put(9,0.5){\line(0,1){9}}
\put(10.1,0.8){$\widetilde{\theta}$}
\put(0.7,9.5){$\widetilde{r}$}
\put(0.4,0.3){$\widetilde{p}$}
\put(0.4,8.2){$\widetilde{q}$}
\put(0.2,4.2){$\widetilde{p}_0$}
\put(5.1,4.2){$\widetilde{q}_0$}
\put(9.1,4.2){$\widetilde{p}_1$}
\put(0,4.9){$a$}
\put(-0.4,8.9){$2a$}
\put(0,0.9){$0$}
\put(1,0){$0$}
\put(5,0){$\pi$}
\put(9,0){$2\pi$}
\put(2.8,5.2){$\phi$}
\put(3.8,8){$\widetilde{\gamma}$}
\put(4,6.3){$\dot{\widetilde{\gamma}}$}
\qbezier(1,5)(4,6)(4.5,8.7)
\put(4.45,8.5){\vector(1,3){0.1}}
\put(1,5){\vector(3,1){4}}

\qbezier(2.4,5.5)(2.6,5.3)(2.4,5)

\end{picture}
      \caption{The angle $\phi$ between tangent vector of $\widetilde{\gamma}$ and $\frac{\partial}{\partial \widetilde{\theta}}|_{ \widetilde\gamma(s)}$.}
      \label{Clairaut}
\end{figure}
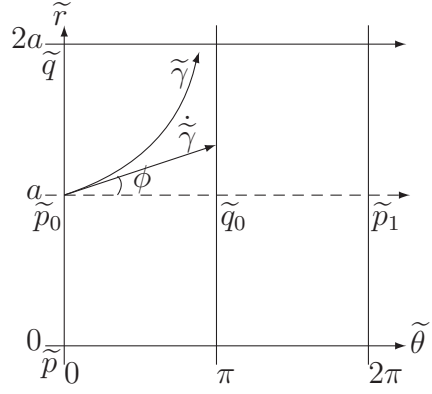

\begin{remark}\label{h-Clairaut particular cases}
\begin{enumerate}
\item Usually, a geodesic 
$\widetilde{\gamma}:[0,l]\to M$, $l>0$,
is determined by its starting point $\widetilde{p}_0\in\widetilde{M}$ and initial velocity  $v:=\widetilde{\gamma}(0)\in T_{\widetilde{p}_0}M$. However, from the Clairaut relation above one can see that this is equivalent to characterize geodesics by the initial point $\widetilde{p}_0$ and Clairaut constant $\nu$. It is customary to use the notation $\widetilde{\gamma}^{\widetilde{p}_0}_{\nu}$.
\item Let $\widetilde{p}_0\in\{\widetilde{r}=a\}$ be a point on the equator, and let $\widetilde\gamma^{\widetilde{p}_0}_{\nu}(s)=(\widetilde r(s),\widetilde \theta(s))$ be the $\widetilde{h}$-geodesic from $\widetilde{p}_0$ with Clairaut constant $\nu$. Observe that 
\begin{enumerate}
\item if $\nu=0$, then $\phi=\pm\frac{\pi}{2}$ and  $\widetilde\gamma^{\widetilde{p}_0}_{0}$ is a meridian, i.e. $\frac{d\widetilde\theta(s)}{ds}=0$;
\item if $\nu=m(a)$, then $\phi=0$ and  $\widetilde\gamma^{\widetilde{p}_0}_{m(a)}$ is a parallel, namely the equator in this case, i.e. $\frac{d\widetilde r(s)}{ds}=0$;
\item if $\nu\in (0,m(a))$, then $\phi\in (-\frac{\pi}{2},\frac{\pi}{2})\setminus\{0\}$, and hence the geodesic  $\widetilde\gamma^{\widetilde{p}_0}_{\nu}(s)=(\widetilde r(s),\widetilde \theta(s))$ is neither a meridian nor a parallel and $\frac{d\widetilde\theta(s)}{ds}>0$.
\end{enumerate}
\end{enumerate}

\end{remark}

By combining the Clairaut relation with (\ref{eq 1}) it follows that the tangent vector along the unit $\widetilde{h}$-geodesic  $\widetilde\gamma^{\widetilde{p}_0}_{\nu}$ has the components
\begin{equation*}
\begin{split}
& \frac{d\widetilde r(s)}{ds}=\pm \sqrt{1-\frac{\nu^2}{m^2(\widetilde r(s))}},\quad   \frac{d\widetilde\theta(s)}{ds}=\frac{\nu}{m^2(\widetilde r(s))}.
\end{split}
\end{equation*}

If we assume $\frac{d\widetilde r(s)}{ds}\neq 0$, for all $s$ in some interval $(s_1,s_2)$, i.e. our geodesic $\widetilde\gamma^{\widetilde{p}_0}_{\nu}$ is not tangent to a parallel, then it follows
\begin{equation}
\widetilde \theta(s_2)-\widetilde \theta(s_1)=\textrm{sign} \frac{d\widetilde r(s)}{ds} \int_{\widetilde r(s_1)}^{\widetilde r(s_2)} \frac{\nu}{m(\tau)\sqrt{m^2(\tau)-\nu^2}}d\tau,
\end{equation}
where $\textrm{sign} \frac{d\widetilde r(s)}{ds} $ is the sign of the component $\frac{d\widetilde r(s)}{ds}$ of the tangent vector. Indeed, if the $\widetilde{h}$-geodesic is not a parallel, then the theorem of implicit functions allows us to write locally $\widetilde\gamma^{\widetilde{p}_0}_{\nu}$ as $\widetilde \theta=\widetilde \theta(\widetilde r)$, for $\widetilde r\in (\widetilde r(s_1),\widetilde r(s_2))$, with the tangent vector 
\begin{equation}\label{dtheta/dr}
\frac{d\widetilde\theta}{d\widetilde r}=\textrm{sign} \frac{d\widetilde r(s)}{ds}
 \frac{\nu}{m(\widetilde r)\sqrt{m^2(\widetilde r)-\nu^2}}.
\end{equation}

Likewise, the $\widetilde{h}$-length of such an $\widetilde\gamma^{\widetilde{p}_0}_{\nu}|_{(s_1,s_2)}$ is given by 
\begin{equation}
\mathcal L_{\widetilde{h}}(\widetilde\gamma^{\widetilde{p}_0}_{\nu}|_{(s_1,s_2)})=\textrm{sign} \frac{d\widetilde r(s)}{ds} \int_{\widetilde r(s_1)}^{\widetilde r(s_2)} 
\frac{m(\tau)}{\sqrt{m^2(\tau)-\nu^2}}d\tau,
\end{equation}
and taking into account the obvious identity
$$
\frac{m(\tau)}{\sqrt{m^2(\tau)-\nu^2}}=
\frac{\sqrt{m^2(\tau)-\nu^2}}{m(\tau)}
+\frac{\nu^2}{m(\tau)\sqrt{m^2(\tau)-\nu^2}},
$$
it follows
\begin{equation}
\mathcal L_{\widetilde{h}}(\widetilde\gamma^{\widetilde{p}_0}_{\nu}|_{(s_1,s_2)})=\textrm{sign} \frac{d\widetilde r(s)}{ds} \int_{\widetilde r(s_1)}^{\widetilde r(s_2)} 
\frac{\sqrt{m^2(\tau)-\nu^2}}{m(\tau)}
d\tau+\nu[\widetilde \theta(s_2)-\widetilde \theta(s_1)].
\end{equation}

Let us assume that $\widetilde \gamma_\nu^{\widetilde{p}_0}(s)=(\widetilde r(s),\widetilde \theta(s))$ is an $\widetilde{h}$-unit speed geodesic from $\widetilde  p_0$, $\{\widetilde r(\widetilde p_0)=a\}$, $\{\widetilde \theta(\widetilde p_0)=0\}$, such that  $\nu\in (0,m(a))$, and 
$\frac{d\widetilde r(s)}{ds}|_{s=0}<0$.
 From Clairaut relation it follows that  $\widetilde{\gamma}^{\widetilde{p}_0}_{\nu}$ must be tangent to the parallel $\{\widetilde r=\xi(\nu)\}$ at a point $\widetilde{\gamma}^{\widetilde{p}_0}_{\nu}(t_1)$ and 
 and return to the equator at $\widetilde{p}_1=\widetilde{\gamma}^{\widetilde{p}_0}_{\nu}(t_0)$, where 
 $$
 t_0=\min\{t>0\ :\ \widetilde r(t)=a\}.
 $$
 Observe that here $\xi:(0,m(a))\to \R$ is the inverse function of $m:[0,b)\to \R$, where $b$ is the smallest value such that $m'|_{[0,b)]}>0$. \\
 
On the universal covering, we can see that 
\begin{equation}\label{half period condition}
\widetilde{\theta}(t_0)-\widetilde{\theta}(0)=2(\widetilde{\theta}(t_0)-\widetilde{\theta}(t_1))
\end{equation}
By integrating \eqref{dtheta/dr} with condition \eqref{half period condition}, it follows (see \cite{BCST}):

\begin{lemma}[Half period function of Riemannian two-sphere of revolution]\label{lemma_h_half}
Let $\widetilde{\gamma}^{\widetilde{p}_0}_{\nu}$ be a $\widetilde{h}$-unit speed geodesic, where $\widetilde{p}_0\in\{\widetilde{r}=a\}$ and $\nu\in(0,m(a))$, i.e. $\widetilde{p}_0$ is a point on equator and $\widetilde{\gamma}^{\widetilde{p}_0}_{\nu}$ is neither meridian nor parallel (equator) (see Figure \ref{lem_h_half}). The $\widetilde{h}$-distance from $\widetilde{p}_0$ to $\widetilde{\gamma}^{\widetilde{p}_0}_{\nu}(t_0)$ is given by the $\widetilde{h}$-half period function 
\begin{equation}\label{h_Half}
\mathcal{H}:(0,m(a))\to \R,\qquad 
\mathcal{H}(\nu)=2\int_{\xi(\nu)}^a\frac{\nu}{m(\tau)\sqrt{m(\tau)^2-\nu^2}}d\tau,
\end{equation}
where $\widetilde{r}=\xi(\nu)$ is the parallel tangent to $\widetilde{\gamma}^{\widetilde{p}_0}_{\nu}(t_0)$.
\end{lemma}
\begin{figure}[H]
\setlength{\unitlength}{0.5cm}
\centering
\begin{picture}(10,10)
\put(1,0.5){\vector(0,1){9}}
\put(0.5,1){\vector(1,0){9.5}}
\put(0.5,9){\vector(1,0){9.5}}
\multiput(0.5,5)(0.5,0){19}{\line(1,0){0.3}}
\put(10,5){\vector(1,0){0.1}}
\put(10.1,0.8){$\widetilde{\theta}$}
\put(0.7,9.5){$\widetilde{r}$}
\put(0.4,0.2){$\widetilde{p}$}
\put(0.4,8.2){$\widetilde{q}$}
\put(5.1,4.2){$\widetilde{q}_0$}
\put(-3.2,4.9){$\widetilde{p}_0=(0,a)$}
\put(10.2,4.9){$\widetilde{p}_1=(2\pi,\widetilde{\gamma}_\nu^{\widetilde{p}_0}(t_0))$}
\put(-0.4,8.9){$2a$}
\put(0,0.9){$0$}
\put(1,0){$0$}
\put(5,0){$\pi$}
\put(9,0){$2\pi$}
\put(5.1,2){$\widetilde{\gamma}_\nu^{\widetilde{p}_0}(t_1)$}
\put(1.1,5.3){$\widetilde{\gamma}_\nu^{\widetilde{p}_0}(0)$}

\put(-1,2.85){$\xi(\nu)$}

\qbezier(1,5)(5,1)(9,4.9)
\put(8.9,4.7){\vector(1,3){0.1}}
\multiput(0.5,2.95)(0.5,0){19}{\line(1,0){0.3}}
\multiput(5,0.5)(0,0.5){19}{\line(0,1){0.3}}
\multiput(9,0.5)(0,0.5){19}{\line(0,1){0.3}}
\put(5.5,0.25){\vector(1,0){3.5}}
\put(4.5,0.25){\vector(-1,0){3.1}}
\end{picture}
\bigskip
      \caption{The Riemannian half period function $\mathcal{H}(\nu)$.}
      \label{lem_h_half}
\end{figure}
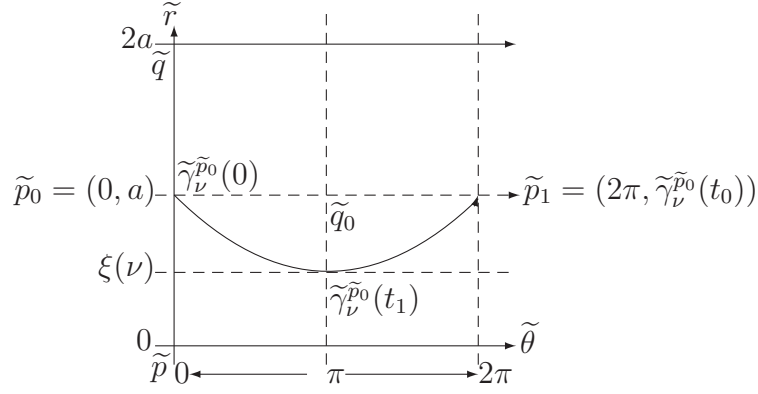

Let $\widetilde{p}_0\in \widetilde{M}$ and $\widetilde{\beta}_\nu(s)$ and $\widetilde{\gamma}_\nu(s)$ for any $\nu\in(0,m(\widetilde{r}(\widetilde{p}_0)))$ denote the geodesic emanating from $\widetilde{p_0}$ with $(\widetilde{r}\circ \widetilde{\beta_\nu})'(0)\geq 0$ and $(\widetilde{r}\circ \widetilde{\gamma_\nu})'(0)\leq 0$. 

Since both geodesics $\widetilde{\beta}_\nu(s)$ and $\widetilde{\gamma}_\nu(s)$
depend smoothly on $\nu\in (0,m(a))$ we obtain two geodesic variations such that all curves in the variation are geodesics. 


We obtain the $h$-Jacobi fields $X_\nu(t)$ and $Y_\nu(t)$ :
\begin{equation*}
X_\nu(t):=\frac{\partial}{\partial \nu} (\widetilde{\beta}_\nu(t)),\quad Y_\nu(t):=\frac{\partial}{\partial \nu} (\widetilde{\gamma}_\nu(t)),
\end{equation*}
along $\widetilde{\beta}_\nu(t)$ and $\widetilde{\gamma}_\nu(t)$. We can see that $X_\nu(0)=Y_\nu(0)=0$.

Let us denote by $\widetilde p_u$ the point of coordinates $(\widetilde r(\widetilde p_u),\widetilde \theta(\widetilde p_u))=(u,0)$, where $u\in (0,2a)$ and $\nu\in (0,m(a))$. Similarly with the case $u=a$, discussed above, for any $\nu \in (0,m(u))$, we consider the $\widetilde{h}$-geodesic $\widetilde \gamma_\nu^{u}$ emanating from $\widetilde p_u$, with Clairaut constant $\nu$ and $(\widetilde{r}\circ \widetilde{\gamma}_\nu^u)'(0)\leq 0$.

The geodesic $\widetilde \gamma_\nu^{u}$ is tangent to the parallel $\{\widetilde r=\xi(u)\}$ at a point $\widetilde \gamma_\nu^{u}(s_0)$, will intersect the equator and then will be tangent to the parallel $\{\widetilde r=2a-\xi(u)\}$ at a point $\widetilde \gamma_\nu^{u}(s_1)$. Clearly, the parameter values $s_0$ and $s_1$ are solutions of the equation $(\widetilde{r}\circ \widetilde{\gamma}_\nu^u)'(s)=0$. Then it is known from the proof of Lemma 2.9 in 
\cite{ST}, or Proposition 7.2.3 in \cite{SST}, that the Jacobi vector field $Y_\nu$ along $\widetilde{\gamma}_\nu^u$ is given by
\begin{equation}\label{h-Jacobi}
Y_\nu(s)=\frac{\partial \widetilde \theta}{\partial \nu}(\widetilde r(s),u,\nu)\left[ 
-\nu\frac{m(\widetilde r(s))}{\sqrt{m^2(\widetilde r(s))-\nu^2}}\left(\frac{\partial}{\partial \widetilde r}\right)_{\widetilde{\gamma}_\nu^u(s)}+
\left(\frac{\partial}{\partial \widetilde \theta}\right)_{\widetilde{\gamma}_\nu^u(s)}
\right].
\end{equation}

Pay attention to the fact that we are using here the parametrization $\widetilde \theta=\widetilde \theta(\widetilde r(s),u,\nu)$ explained above. Some straightforward computations show that $\widetilde \theta(\widetilde r,u,\nu)$ given by 
\begin{equation}
\widetilde \theta(\widetilde r,u,\nu)=\mathcal H(\nu)-\int_{\widetilde r}^{2a-u}
\frac{\nu}{m(\tau)\sqrt{m^2(\tau)-\nu^2}}d\tau,
\end{equation}
where $\mathcal H$ is the Riemannian half-period function in Lemma \ref{lemma_h_half}.

A point $q_0:=\widetilde \gamma_\nu(l)$, where $\nu\in(-m(\widetilde{r}(p_0)),m(\widetilde{r}(p_0)))$ and $l>0$, is {\it $\widetilde{h}$-conjugate} to $\widetilde{p}_0$ along $\widetilde \gamma_\nu(t)$ if and only if $Y_\nu(l)=0$, and taking into account that  $\Bigl(\frac{\partial}{\partial \widetilde r}\Bigr)_{\widetilde{\gamma}_\nu^u(s)},
\Bigl(\frac{\partial}{\partial \widetilde \theta}\Bigr)_{\widetilde{\gamma}_\nu^u(s)}$ are linear independent vectors on $T_{\widetilde{\gamma}_\nu^u(s)} M$, one obtains the differential equation 
\begin{equation}
\frac{\partial \widetilde \theta}{\partial \nu}(\widetilde r,u,\nu)=0
\end{equation}
along $\widetilde{\gamma}_\nu^u(s)$.

It can be shown that this differential equation has a unique solution 
$\widetilde r(s_c, \nu, u)$ that is the $\widetilde r$ coordinate of the $\widetilde{h}$-first conjugate point of $\widetilde p_u$ on 
$\widetilde{\gamma}_\nu^u(s)$.

The $\widetilde \theta$ coordinate of the $\widetilde{h}$-first conjugate point is obtained by substitution $\widetilde \theta(s_c, u, \nu)=\widetilde \theta(\widetilde r(s_c, u, \nu),u, \nu)$.

\begin{definition}
Let $\gamma:[0,t_0]\to M$ be a minimal $h$-geodesic segment on a complete Riemannian manifold $(M,h)$. The endpoint $\gamma(t_0)$ of the geodesic segment is called a {\it $h$-cut point} of $p:=\gamma(0)$ along $\gamma$ if any extended geodesic segment $\gamma^*:[0,t_1]\to M$ of $\gamma$, where $t_1>t_0$, is not a minimizing arc joining $p$ to $\gamma^*(t_1)$ anymore. The {\it $h$-cut locus} $\mathcal{C}^h_p$ of the point $p$ is defined by the set of the cut points along all geodesics segments emanating from $p$.
\end{definition}

The structure of the $h$-cut locus $\mathcal{C}^h_p$ of $(M,h)$ was obtained in \cite{ST}.
\begin{theorem}[\cite{ST}]\label{thm_h_cut_locus}
Let $(M,dr^2+m(r)^2d\theta^2)$ be a $2$-sphere of revolution with a pair of poles $p,q$ and satisfying properties in Remark \ref{rem_2_sphere_properties}. Then the $h$-cut locus of a point $x\in M \setminus\{p,q\}$ with $\{\theta(x)=0\}$ is

\begin{itemize}
\item[1.] The antipodal point $\mathcal{C}^h_x=(2a-r(x),\pi)$, when $G(x)$ is a positive constant.
\item[2.] A subarc of the opposite half meridian $\mathcal{C}^h_x\subset\{\theta=\pi\}$, when $G(x)$ is monotone non-increasing along meridian from the pole $p$ to the point on $\{r=a\}$.
\item[3.] A subarc of the antipodal parallel $\{r=2a-r(x)\}$, that is $\mathcal{C}^h_x=r^{-1}(2a-r(x))\cap\theta^{-1}(\mathcal{H}(m(r(x))), 2\pi-\mathcal{H}(m(r(x))))$, when $G(x)$ is monotone non-decreasing along meridian from the pole $p$ to the point on $\{r=a\}$.
\end{itemize} 
\end{theorem}

\subsection{Randers rotational metrics}\label{Ran_rota}

In a previous paper \cite{HCS} we have constructed a Randers rotational metric on a surface of revolution homeomorphic to $\R^2$. We will construct a Randers rotational metric on a 2-sphere of revolution in a similar manner in the following. 

Let $(M,h)$ be the $2$-sphere of revolution considered in the previous section. 
Observe that there exists a constant $\mu<\{\frac{1}{\max\{m(r)\}}:r\in[0,2a]\}$, such that $\mu<\frac{1}{m(r)}$ for any $r\in [0,2a]$.

\begin{proposition}\label{Prop_Randers_metric}
If $(M,h)$ is a surface of revolution 
and $W=\mu\cdot\frac{\partial}{\partial \theta}$ is a breeze on $M$ blowing along parallels, then the Randers metric $(M,F=\alpha+\beta)$ obtained by the Zermelo's navigation process with data $(h,W)$ is a Finsler metric 
on $M$, where $\alpha=\sqrt{a_{ij}(x)y^iy^j}$, $\beta=b_i(x)y^i$ are defined by
\begin{equation}\label{Randers_metrics}
(a_{ij})=\left(
\begin{array}{cc}
\frac{1}{1-{\mu}^2m^2} & 0\\
0 & \frac{m^2}{\left(1-{\mu}^2m^2\right)^2}
\end{array}\right)
\text{, }
b_i=\left(
\begin{array}{c}
0 \\
-\frac{\mu m^2}{1-{\mu}^2m^2}
\end{array}\right),\quad i,j=1,2.
\end{equation}
\end{proposition}

Indeed, observe that due to our condition $\mu <\frac{1}{m(r)}$ for all $r\in[0,a]$, $W$ in the canonical basis  $(\frac{\partial}{\partial r},\frac{\partial}{\partial \theta})$ of $T_xM$, reads
 $W=(W^1,W^2)=(0,{\mu})$, and hence 
$h(W,W)=b^2=\left({\mu}m\right)^2<1$, where $b^2:=a^{ij}b_ib_j$ is the Riemannian $a$-norm of the covariant vector $b=(b_1,b_2)$. This condition guarantees the strong convexity of the Randers metric $F=\alpha+\beta$ (see \cite{BCS}).


It is trivial to see that $W$ is a Killing vector field of $(M,h)$, and taking into account that the flow of $W$ is $\varphi(s; r(s),\theta(s))=(r(s),\theta(s)+\mu\cdot s)$, we obtain the global characterization of $F$-geodesics.

\begin{proposition}\label{F_Geodesic}

Let $(M,F=\alpha+\beta)$ be
the Randers rotational metric constructed from the navigation data $(h,W)$, where $(M,h)$ is a Riemannian 2-sphere  of revolution, and $W=\mu\cdot\frac{\partial}{\partial \theta}$, $\mu<\{\frac{1}{\max\{m(r)\}}:r\in[0,2a]\}$, is the breeze on $M$ blowing along parallels, then the $F$-unit speed geodesics $\mathcal{P}:(-\epsilon,\epsilon)\to M$ are given by
\begin{equation}\label{global Finsler geodesics}
\mathcal{P}(s)=(r(s),\theta(s)+{\mu}s),
\end{equation}
where $\gamma(s)=(r(s),\theta(s))$ is an $h$-unit speed geodesic.
\end{proposition}

Indeed, taking into account that Zermelo's navigation gives 
\begin{equation}\label{h=1 then F=1}
h(\gamma(s),\dot{\gamma}(s))=1 \textrm{   if and only if    } 
F(\mathcal P(s),\dot{\mathcal P}(s))=1.
\end{equation}
It follows that we can use the same arclength parameter $s$ on both Riemannian and Randers geodesics, and since $W$ is $h$-Killing vector field, the conclusion follows from \cite{R}, or can be verified by straightforward computation.


\begin{corollary}
The pair $(M,F)$ is a forward complete Finsler surface of Randers type. 
\end{corollary}

We recall from our previous work \cite{HCS} the {\it Finsler version of Clairaut relation}. For an $F$-unit geodesic $\mathcal P(s)=\varphi(s,\gamma(s))$ obtained by deviating an $h$-geodesic $\gamma(s)$ with Clairaut constant $\nu$ by means of the $W$-flow $\varphi$, the following relation holds
\begin{equation}\label{Finslerian Clairaut rel}
\cos \psi(s)=\frac{\nu+\mu m^2(r(s))}{m(r(s))\sqrt{1+2\mu\nu+\mu^2 m^2(r(s))}},
\end{equation}
where $\psi(s)$ is the angle between the vectors $\dot{\mathcal P}(s)$ and $\frac{\partial}{\partial\theta}|_{\mathcal P(s)}$. We have proved this formula for a surface of revolution homeomorphic to $\R^2$ in \cite{HCS}, but the proof carries out on any kind of surface of revolution. 

\begin{figure}[H]
\begin{center} \setlength{\unitlength}{1cm} 
\begin{picture}(10,4.5) 
\put(4,1){\vector(1,0){5}}
\put(5,0){\vector(0,1){5}} 
\put(4,1){\vector(1,0){3.1}} 
\put(5,1){\vector(2,1){2.8}} 
\put(5,1){\vector(1,2){0.7}}
\qbezier(7,1)(7.025,1.05)(7.05,1.1) 
\qbezier(7.1,1.2)(7.125,1.25)(7.15,1.3)
\qbezier(7.2,1.4)(7.225,1.45)(7.25,1.5)
\qbezier(7.3,1.6)(7.325,1.65)(7.35,1.7)
\qbezier(7.4,1.8)(7.425,1.85)(7.45,1.9)

\qbezier(7.5,2)(7.525,2.05)(7.55,2.1)
\qbezier(7.6,2.2)(7.625,2.25)(7.65,2.3)
\multiput(5,2.4)(0.2,0){13} {\line(1,0){0.1}} 
\qbezier(5.2,1.5)(5.7,1.6)(5.7,1) 
\qbezier(6,1.5)(6.5,1.6)(6.5,1) 
\put(5.5,1.7){$\phi$} 
\put(6.5,1.5){$\psi$} 
\put(4.5,4.5){$\frac{\partial}{\partial r}$} 
\put(6.5,0.5){${W}={\mu \cdot\frac{\partial}{\partial \theta}}$}
\put(9,0.4){$\frac{\partial}{\partial \theta}$} 
\put(5.7,2.7){$\dot{\gamma}$} 
\put(7.8,2.6){$\dot{\mathcal{P}}$} 
\end{picture} 
\end{center}
\caption{The angle $\psi$  between $\dot{\mathcal P}$ and a parallel.}
\end{figure}
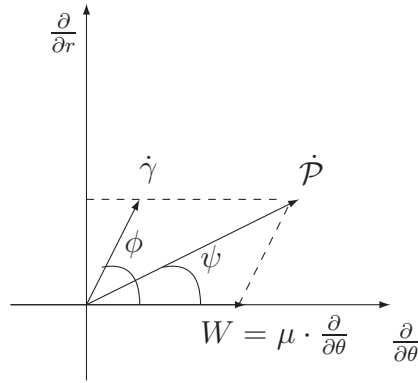

\begin{remark}
\begin{enumerate}
\item We have seen in Remark \ref{h-Clairaut particular cases} that an $h$-geodesic is characterised by its initial point and Clairaut constant $(p_0,\nu)$, that is equivalent to the usual initial conditions $(p_0,v)\in TM$. Since the corresponding Finsler geodesic is also determined by its starting point ${p}_0$ and initial velocity  $y:=v+\mu\cdot\frac{\partial}{\partial \theta}\in T_{p_0}M$, where $\mu$ is constant, we can see that this $F$-geodesic is uniquely determined by its initial point and Clairaut constant $\nu$. We have to pay attention though that $\nu$ is the Clairaut constant of the original $h$-Riemannian geodesic. 
\item Let $p_0\in\{r=a\}$ be a point on the equator, let $\gamma^{p_0}_{\nu}(s)=( r(s),\theta(s))$ be the $h$-geodesic from $p_0$ with Clairaut constant $\nu$, and let $\mathcal {P}(s)=( r(s), \theta(s)+\mu s)$
be the corresponding $F$-geodesic. Observe that 
\begin{enumerate}
\item $\mathcal P$ is a meridian, that is $\psi=\pm\frac{\pi}{2}$, if and only if $\nu=-\mu m^2(a)$. 

Indeed, $\psi=\pm\frac{\pi}{2}$ means $\cos \psi =0$, and from Finslerian Clairaut relation \eqref{Finslerian Clairaut rel} we obtain the desired value.

\item  $\mathcal P$ is a parallel,  namely the equator in this case, that is $\psi=0$, if and only if $\nu=m(a)$.

\item if $\nu\in (-\mu m^2(a),0)\cup(0,m(a))$, then $\psi\in (-\frac{\pi}{2},\frac{\pi}{2})\setminus\{0\}$, and the geodesic  $\mathcal{P}^{p_0}_{\nu}(s)=(r(s), \theta(s)+\mu s)$ is neither a meridian nor a parallel with $\frac{d\theta(s)}{ds}>0$.
\end{enumerate}
\end{enumerate}

\end{remark}

We have the following important result. 

\begin{lemma}\label{F_curvature}
The flag curvature $\mathcal{K}$ of the Randers rotational metric $(M,F=\alpha+\beta)$ given by \eqref{Randers_metrics} lives on the base manifold $M$. Moreover, we have $\mathcal{K}(x,y)=\mathcal{K}(x)=G(x)$, for any $(x,y)\in TM$, where $G$ is the Gaussian curvature of $(M,h)$.
\end{lemma}
\begin{proof}
Even though similar with the proof of Lemma 4.3 in \cite{HCS} we sketch it here for the sake of completeness. We can see that our Randers rotational surface of revolution is Finsler-Einstein with Ricci scalar $Ric^{(F)}=\mathcal{K}(x)$, where $\mathcal{K}(x)$ is the sectional curvature of $(M,F)$.

From \cite{BR}, we know that $(M,F)$ is Finsler-Einstein with Ricci scalar $Ric^{(F)}=\mathcal{K}(x)$ if and only if $(M,h)$ is Einstein with Ricci scalar $Ric^{(h)}=\mathcal{K}(x)$ and $W$ is Killing vector field for $(M,h)$.

Next, we recall that any Riemannian surface $(M,h)$ is an Einstein manifold with Ricci scalar $Ric^{(h)}=G(x)$ that completes the proof.

$\qedd$
\end{proof}
We turn now to the study of the conjugate points of $F$-geodesics. 
\begin{proposition}\label{Prop_F_con}
Let $(M,F=\alpha+\beta)$ be a Randers rotational surface of revolution with navigation data $(h,W)$, where $W=\mu\cdot\frac{\partial}{\partial\theta}$ is the breeze on $M$ blowing along parallels $\mu<\frac{1}{m(r)}$ for any $r$. Suppose that $\gamma:[0,l]\to M$ is an $h$-geodesic and $\mathcal{P}(s)=\varphi(s,\gamma(s))$ is the corresponding $F$-geodesic, $t\in[0,l]$. Then $\mathcal{P}(l)$ is conjugate to $p=\mathcal{P}(0)$ along $\mathcal{P}$ (with respect to metric $F$) if and only if $\gamma(l)$ is conjugate to $p=\gamma(0)$ along $\gamma$ (with respect to metric $h$).
\end{proposition}

\begin{proof}
Let $\gamma:[0,l]\to M$ be an $h$-unit speed geodesic. Suppose $\Gamma(t,s):(-\varepsilon,\varepsilon)\times[0,l]\to M$ be an $h$-geodesic variation of $\gamma(s):=\Gamma(0,s)$ with variation vector field
\begin{equation*}
J(s):=\frac{\partial\Gamma(t,s)}{\partial t}\Bigg\vert_{t=0}.
\end{equation*}

Observe that this $J$ is actually given by \eqref{h-Jacobi} for any $\nu\in(0,m(a))$. If we assume that $\gamma(l)$ is $h$-conjugate to $\gamma(0)$ it follows that
\begin{equation*}
J(0)=J(l)=0 \quad \text{and} \quad J(s)\neq 0, \quad s\in(0,l).
\end{equation*}

By using the wind $W$, blowing up on $M$, with the flow $\varphi$, by deviating $\gamma$ we obtain the corresponding $F$-geodesic $\mathcal{P}(s)=\varphi(s,\gamma(s))$.\\
Let us consider the $F$-geodesic variation
\begin{equation*}
\overline{\mathcal{P}}(t,s)=\varphi(v(t)s,\Gamma(t,s)),
\end{equation*}
where $v(t)$ is the constant $h$-speed of the geodesic variation $\Gamma(t,s)$.

We obtained the $F$-Jacobi field
\begin{equation*}
\begin{split}
\mathcal{J}(s)=d\varphi\cdot J(s).
\end{split}
\end{equation*}
If we consider the flow $\varphi=(\varphi_1,\varphi_2)=(r,\theta+\mu s)$, it follows
\begin{equation}\label{iden_matrix}
d\varphi=\left(\begin{matrix}
\frac{\partial\varphi_1}{\partial r} & \frac{\partial\varphi_1}{\partial \theta} \\
\frac{\partial\varphi_2}{\partial r} & \frac{\partial\varphi_2}{\partial \theta} \\
\end{matrix}\right)=
\left(\begin{matrix}
1 & 0 \\ 0 & 1 
\end{matrix}\right),
\end{equation}
that is the identity matrix.

We obtain that $\mathcal{J}$ vanishes if and only if $J$ does, hence
\begin{equation*}
\mathcal{J}(0)=\mathcal{J}(l)=0 \quad \text{and} \quad \mathcal{J}(s)\neq 0, \quad s\in(0,l),
\end{equation*}
that is $\mathcal{P}(l)$ is conjugate to $p=\mathcal{P}(0)$ along $\mathcal{P}$, whenever $\gamma(l)$ is conjugate to $p=\gamma(0)$ along $\gamma$.

$\qedd$
\end{proof}

\section{The proof of Theorem \ref{thm_F_cut_locus}}\label{Proof of main theorem}

\quad Let $(M,F=\alpha+\beta)$ be a Randers rotational $2$-sphere of revolution obtained from the navigation data $(h,W)$, where $h$ is the Riemannian metric of the $2$-sphere of revolution $M$, and $W:=\mu\cdot\frac{\partial}{\partial \theta}$ is the wind blowing along the parallels, where $\mu<\{ \frac{1}{\max\{m(r)}:r\in[0,2a]\}$.

Extending by analogy the definitions of poles from Riemannian setting (see Section \ref{Riem_2_sph}), we obtain

\begin{lemma}
The $F$-cut point of the pole $p$ on $(M,F)$ is the other pole $q$.
\end{lemma}
\begin{proof}
Recall that a pair of poles $p,q$ on $(M,h)$ are invariant under the flow acting along parallels, i.e. $\varphi(s,p)=p$ and $\varphi(s,q)=q$, for any $s\in\R$. \\
Since the $h$-geodesics joining $p$ and $q$ are meridians and all of them have same $h$-length, it follows that the $F$-geodesics joining $p$ and $q$ are bending meridians with same $F$-length by \eqref{h=1 then F=1}. Hence we get that $q$ is the cut point of $p$ on $(M,F)$.
$\qedd$
\end{proof}

\begin{remark}
In \cite{HCS} the curve $\mathcal{P}(s)=\varphi(s,\gamma(s))$ is called a {\it twisted meridian}, where $\gamma(s)$ is a meridian. However since a pair of poles on $2$-sphere of revolution $M$ are invariant under the wind, we prefer to use the words {\it bending meridian} for $\mathcal{P}(s)=\varphi(s,\gamma(s))$, where $\gamma(s)$ is meridian on $(M,h)$.
\end{remark}

\begin{corollary}
The points $p,q$ are a pair of poles on $(M,F)$.
\end{corollary}

\begin{remark}
In the Finslerian universal covering manifold $(\widetilde{M},\widetilde{F}=\widetilde{\alpha}+\widetilde{\beta})$, with the covering projection $\Pi:\widetilde{M}\to M \setminus\{p,q\}$ we use the notation $\widetilde{\mathcal{P}}^+(s)=(s,\widetilde{\gamma}(s))$ for an $\widetilde{F}$-geodesic obtained from $\widetilde{\gamma}(s)$ in the wind blowing direction and $\widetilde{\mathcal{P}}^-(s)=(-s,\widetilde{\gamma}(s))$ for an $\widetilde{F}$-geodesic advancing against the wind.
\end{remark}

\begin{lemma}\label{F-equator}
Let $\widetilde{\gamma}(s)=(\widetilde{r}(s),\widetilde{\theta}(s))$ be an $\widetilde{h}$-unit speed geodesic on $\widetilde{M}$ with Clairaut constant $\nu=m(a)$ joining the points $\widetilde{p}_0:=(a,0)$ and $\widetilde{q}_0:=(a,\pi)$ , i.e. $\widetilde{\gamma}(s)$ is an equator and $\widetilde{\theta}(\widetilde{p}_0)=0$, $\widetilde{\theta}(\widetilde{q}_0)=\pi$ or $\widetilde{q}_0$ is antipodal point of $\widetilde{p}_0$ along $\widetilde{\gamma}$. \\
 Then the $\widetilde{F}$-unit speed geodesic $\widetilde{\mathcal{P}}^+(s)=\varphi(s,\widetilde{\gamma}(s))$ will join the point $\widetilde{p}_0=\widetilde{\mathcal{P}}^+(0)$ with $\widetilde{q}_1=\widetilde{\mathcal{P}}^+(\pi)=(a,\pi(1+\mu))$. On the other hand, $\widetilde{\mathcal{P}}^-(s)=\varphi(-s,\widetilde{\gamma}(s))$ will join $\widetilde{p}_0=\widetilde{\mathcal{P}}^-(0)$ to the point $\widetilde{q}_2=\widetilde{\mathcal{P}}^-(\pi)=(a,\pi(1-\mu))$.
\end{lemma}
\begin{remark}
Observe that $\Pi(\widetilde{q}_1)=\Pi(\widetilde{q}_2)\in M$.
\end{remark}
\begin{proof}
Let $\widetilde{p}_0$ be an arbitrary point on equator and $\widetilde{q}_0$ be an antipodal point to $\widetilde{p}_0$. Let $\widetilde{\gamma}(s)=(\widetilde{r}(s),\widetilde{\theta}(s))$, be an $\widetilde{h}$-unit speed geodesic joining $\widetilde{p}_0$ to $\widetilde{q}_0$, that is
\begin{equation*}
\widetilde{p}_0=\widetilde{\gamma}(0)=(a,0),\quad \widetilde{q}_0=\widetilde{\gamma}(\pi)=(a,\pi). 
\end{equation*}

We recall that the wind is blowing along the parallels (see \cite{HCS}). Since $\widetilde{d}_h(\widetilde{p}_0,\widetilde{q}_0)=\pi$, we know that the $\widetilde{F}$-unit speed geodesic $\widetilde{\mathcal{P}}^+(s)=\varphi(s,\widetilde{\gamma}(s))$ obtained by $\widetilde{\gamma}(s)$ is joining $\widetilde{p}_0$ to the point
\begin{equation*}
\begin{split}
\widetilde{\mathcal{P}}^+(\pi)&=\varphi(\pi,\widetilde{\gamma}(\pi))=(a,\pi(1+\mu)),\\
\end{split}
\end{equation*}
therefore the point $\widetilde{q}_0$ will change the position to $\widetilde{q}_1=(a,\pi(1+\mu))$, hence $\widetilde{d}_F(\widetilde{p}_0,\widetilde{q}_1)=\pi$. 

On the other hand $\widetilde{\mathcal{P}}^-(s)$ will join $\widetilde{p}_0$ to $\widetilde{q}_2=(a,\pi(1-\mu))$.

$\qedd$

\end{proof}

\begin{remark}\label{lemna_F_half}
Let $\widetilde{\mathcal{P}}^{\widetilde{p}_0}_\nu(s)=(\widetilde{r}(s),\widetilde{\theta}(s)+\mu s)$ be an $\widetilde{F}$-unit speed geodesic 
 emanating from $\widetilde{p}_0\in\{r=a\}$ and $\nu\in(0,m(a))$. One can see that
\begin{equation*}
(\widetilde{\mathcal{P}}^{\widetilde{p}_0}_\nu)^2(r(b))-(\widetilde{\mathcal{P}}^{\widetilde{p}_0}_\nu)^2(r(a))=\int^{r(b)}_{r(a)}\left(\frac{d\theta}{dr}+\mu\frac{ds}{dr}\right)dr.
\end{equation*}
We know that $\widetilde{\mathcal{P}}_\nu(s):=\widetilde{\mathcal{P}}^{\widetilde{p}_0}_\nu$ must be tangent to the parallel $\xi(\nu)$ at $\widetilde{\mathcal{P}}_\nu(t_1)$ and then return to the equator at $\widetilde{\mathcal{P}}_\nu(t_0)$ (see Figure \ref{rem_F_half}). Then, by a similar computation as in the Riemannian case, the $F$-distance from $\widetilde{p}_0$ to $\widetilde{\mathcal{P}}_\nu(t_0)$ in the wind direction is given by the following $F$-{\it half period function}  
\begin{equation}\label{F_Half_front}
\mathcal{H}_F^+(\nu)=\mathcal{H}(\nu)+\psi(\nu),
\end{equation}
where $\psi(\nu):=2\mu(a-\xi(\nu))$, and $\mathcal{H}$ is the $h$-half period function (see \eqref{h_Half}).
For the direction against the wind we obtain 
\begin{equation}\label{F_Half_back}
\mathcal{H}_F^-(\nu)=\mathcal{H}(\nu)-\psi(\nu),
\end{equation}
see \cite{HK} for computational details.
\begin{figure}[H]
\setlength{\unitlength}{0.5cm}
\centering
\begin{picture}(20,10)
\put(1,0.5){\vector(0,1){9}}
\put(0.5,1){\vector(1,0){19.5}}
\put(0.5,9){\vector(1,0){19.5}}
\multiput(0.5,5)(0.5,0){39}{\line(1,0){0.3}}
\put(20,5){\vector(1,0){0.1}}
\put(20.1,0.8){$\widetilde{\theta}$}
\put(0.7,9.5){$\widetilde{r}$}
\put(0.4,0.2){$\widetilde{p}$}
\put(0.4,8.2){$\widetilde{q}$}
\put(-3.2,4.9){$\widetilde{p}_0=(0,a)$}
\put(9.2,5.4){$\widetilde{\gamma}_\nu^{\widetilde{p}_0}(t_0)$}
\put(13.2,5.3){$\widetilde{\mathcal{P}}_\nu(t_0)$}
\put(-0.4,8.9){$2a$}
\put(0,0.9){$0$}
\put(4.5,-0.3){$\mathcal{H}(\nu)$}
\put(7,-1){$\mathcal{H}^+_F(\nu)$}

\put(-1,2.85){$\xi(\nu)$}

\qbezier(1,5)(5,1)(9,4.9)
\put(8.9,4.7){\vector(1,3){0.1}}

\qbezier(1,5)(7,1)(13,4.9)
\put(12.9,4.7){\vector(1,3){0.1}}

\multiput(0.5,2.95)(0.5,0){39}{\line(1,0){0.3}}
\multiput(5,0.5)(0,0.5){19}{\line(0,1){0.3}}
\multiput(9,0.5)(0,0.5){19}{\line(0,1){0.3}}
\multiput(13,0.5)(0,0.5){19}{\line(0,1){0.3}}
\put(6.3,0){\vector(1,0){2.7}}
\put(4.5,0){\vector(-1,0){3.3}}
\put(9.5,-0.7){\vector(1,0){3.5}}
\put(6.8,-0.7){\vector(-1,0){5.6}}
\end{picture}
\bigskip
      \caption{The $h$-half period function and $F$-half period function.}
      \label{rem_F_half}
\end{figure}
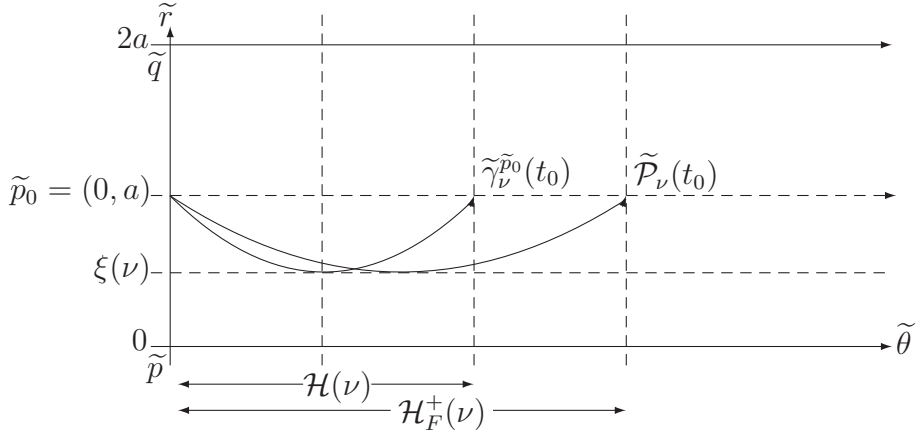

If $m'|_{[0,a)}\neq 0$ then we can assume $m'>0$ on $(0,a)$, in this case, taking into account that $\xi(\nu)=(m|_{(0,a)})^{-1}$ observe that the function $\psi(\nu)=2\mu(a-\xi(\nu))$ is decreasing function when $\xi(\nu)\in(0,a)$ and increasing when $\xi(\nu)\in(a,2a)$. 
\end{remark}

\begin{proposition} \label{prop_F_cut}
Let $x\in M\setminus\{p,q\}$ be an arbitrary point. Then $q_0$ is an $F$-cut point to $x$ on $\mathcal{P}$ if and only if $\hat{q}_0$ is $h$-cut point to $x$ on $\gamma$, where $\mathcal{P}(s)=\varphi(s,\gamma(s))$ is the corresponding $F$-geodesic obtained from $\gamma$, $\mathcal{P}(0)=\gamma(0)=x$.
\end{proposition}

\begin{proof}
Let $\gamma:[0,l]\to M$ be an $h$-unit minimizing geodesic from $x$ to $\hat{q}_0=\gamma(l)$ and $\hat{q}_0$ is a $h$-cut point of $x$, i.e. $\hat{q}_0\in \mathcal{C}^h_x$.

Let $\mathcal{P}(s)$ be the $F$-unit geodesic obtain from $\gamma(s)$ and let $q_0:=\mathcal{P}(l)$.\\
Assume $q_0$ is not $F$-cut point of $x$ on $\mathcal{P}$, that is there exists a shorter minimizing $F$-geodesic $\mathcal{P}_0:[0,l_0]\to M$ from $x=\mathcal{P}_0(0)$ to $q_0=\mathcal{P}_0(l_0)$ where $d_F(x,q_0):=l_0<l$.

From $\mathcal{P}_0$, we construct the corresponding $h$-geodesic
\begin{equation*}
\gamma_0:[0,l_0]\to M,\quad \gamma_0(s)=\varphi(-s,\mathcal{P}(s)),
\end{equation*}
where $\gamma_0(0)=\mathcal{P}_0(0)=x$ and $\gamma_0(l_0)=\varphi(-l_0,\mathcal{P}_0(l_0))=\varphi(-l_0,q_0)=\varphi_{q_0}(\l_0)$.

Let us denote by $\zeta$ the curve
\begin{equation*}
\zeta:[-l_0,-l]\to M,\quad \zeta(s)=\varphi(s,q_0),
\end{equation*}
(see Figure \ref{proof of prop F cut}).

Then, from triangle inequality, we have
\begin{equation}\label{main_lem_1}
\mathcal{L}_h(\zeta)\geq \mathcal{L}_h(\gamma)-\mathcal{L}_h(\gamma_0).
\end{equation}

On the other hand, we compute $\mathcal{L}_h(\zeta)$ as follows
\begin{equation*}
\Vert\dot{\zeta}(s)\Vert^2_h=\Vert W_{\varphi(s,q_0)}\Vert^2_h=\Vert d\varphi(W_{q_0})\Vert^2_h=\Vert W_{q_0}\Vert^2_h=(\mu m(r(q_0)))^2<1,
\end{equation*}
where $d\varphi$ is identity from \eqref{iden_matrix}.

It follows that
\begin{equation}\label{main_lem_2}
\mathcal{L}_h(\zeta)=\int_{-l}^{-l_0} \Vert \dot{\zeta}(s)\Vert_h ds<\int_{-l}^{-l_0}ds=l-l_0=\mathcal{L}_h(\gamma)-\mathcal{L}_h(\gamma_0).
\end{equation}
From \eqref{main_lem_1} and \eqref{main_lem_2}, we get a contradiction, hence $q_0$ is an $F$-cut point of $x$ along $\mathcal{P}$.
 
$\qedd$

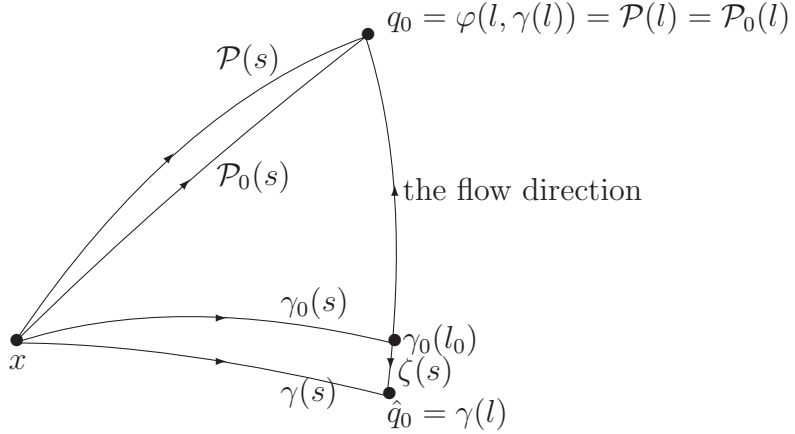
\begin{figure}[H]
\setlength{\unitlength}{0.7cm}
\centering
\begin{picture}(10,10)
\put(-0.1,1.9){$\bullet$}
\put(6.9,0.9){$\bullet$}
\put(6.5,7.7){$\bullet$}
\qbezier(0,2)(3,6.5)(6.6,7.8)
\qbezier(0,2)(3,5)(6.6,7.8)
\qbezier(0,2)(3,2)(7,1)
\qbezier(0,2)(3,3)(7.1,2)
\qbezier(6.6,7.8)(7.5,5)(7,1)
\put(7,1.9){$\bullet$}
\put(7.15,5){\vector(0,1){0}}
\put(7.3,4.7){the flow direction}
\put(7.2,1.3){$\zeta(s)$}
\put(-0.1,1.5){$x$}
\put(7,0.5){$\hat{q}_0=\gamma(l)$}
\put(7.3,1.9){$\gamma_0(l_0)$}
\put(7,8){$q_0=\varphi(l,\gamma(l))=\mathcal{P}(l)=\mathcal{P}_0(l)$}
\put(5,2.6){$\gamma_0(s)$}
\put(5,0.9){$\gamma(s)$}
\put(3.8,7.2){$\mathcal{P}(s)$}
\put(3.8,5){$\mathcal{P}_0(s)$}
\put(4,1.65){\vector(1,0){0}}
\put(4,2.47){\vector(1,0){0}}
\put(3,5.6){\vector(1,1){0}}
\put(3.3,5.1){\vector(1,1){0}}
\put(7.07,1.5){\vector(0,-1){0}}
\end{picture}
      \caption{The proof of Proposition \ref{prop_F_cut}.}
      \label{proof of prop F cut}
\end{figure}

\end{proof}
Here is the proof of our Theorem \ref{thm_F_cut_locus}

{\it Proof of Theorem \ref{thm_F_cut_locus}.}
Let $x\in M\setminus\{p,q\}$.
We recall that the flow for navigation data is $\varphi(s;r(s),\theta(s))=(r(s),\theta(s)+\mu s)$. Propositions \ref{Prop_F_con} and \ref{prop_F_cut} imply that $F$-cut locus is corresponding to the $h$-cut locus.
\begin{enumerate}
\item The case when $\mathcal{K}$ is monotone non-increasing.

\quad  In Riemannian case the $h$-cut locus $\mathcal{C}^h_x$ of $x$, when the Gaussian curvature is monotone non-increasing, is a subarc of the opposite half meridian $\{\theta=\pi\}$, which are denote by $\tau_x\vert_{[c,2a-c]}$, where $\tau_x(c)$ is the $h$-first conjugate point of $x$ along $\tau_x$.\\
Therefore by taking into account Propositions \ref{Prop_F_con} and \ref{prop_F_cut} the $F$-cut locus is the following subarc of the opposite half bending meridian of $x$:
\begin{equation*}
\mathcal{C}^F_x=\varphi(d(x,\tau(t)),\tau(t)),\quad t\in[c,2a-c].
\end{equation*}

\item The case when $\mathcal{K}$ is monotone non-decreasing.

\quad In the Riemannian case (see \cite{ST}), if the Gaussian curvature $G$ is monotone non-decreasing then the $h$-cut locus $\mathcal{C}^h_x$ of $x$ is a subarc of the antipodal parallel $\{r=2a-r(x)\}$, that is 
\begin{equation*}
\mathcal{C}^h_x=r^{-1}(2a-r(x))\cap \theta^{-1}\{\mathcal{H}(m),2\pi-\mathcal{H}(m)\},
\end{equation*}
where $\mathcal{H}$ is $h$-half period function defined in \eqref{h_Half} and $m:=m(r(x))$.\\
Next, let $\hat{q}_0$ be the $h$-first conjugate point of $x$ on front side, i.e.
\begin{equation*}
\hat{q}_0=(2a-r(x),\mathcal{H}(m)),
\end{equation*}
and recall that our wind is blowing along the parallels, therefore the $F$-first conjugate point to $x$ is
\begin{equation*}
r^{-1}(2a-r(x))\cap \theta^{-1}\{\mathcal{H}(m)+\psi(x)\},
\end{equation*}
where $\psi(x)=\mu \cdot d(x,\hat{q}_0)$. On the other hand the $F$-first conjugate point to $x$ on the back side is
\begin{equation*}
r^{-1}(2a-r(x))\cap \theta^{-1}\{2\pi-(\mathcal{H}(m)-\psi(x))\},
\end{equation*}
hence we obtain
\begin{equation*}
\mathcal{C}_x^F=r^{-1}(2a-r(x))\cap \theta^{-1}\{\mathcal{H}(m)+\psi(x),2\pi-(\mathcal{H}(m)-\psi(x))\}.
\end{equation*}

\item The case when $\mathcal{K}$ is constant.

\quad Let $M$ be the round sphere of radius $R$. Recall that in the Riemannian case when $G=\frac{1}{R^2}$ is constant, the cut locus of any point $x\in M\setminus\{p,q\}$ is its antipodal point, i.e. $\mathcal{C}^h_x=\hat{q}_0=(2a-r(x),\pi)$, where $\theta(x)=0$. Since $d_h(x,\hat{q}_0)$ is equal to the half of circumference, i.e. $d_h(x,\hat{q}_0)=\pi\cdot R$, from Proposition \ref{prop_F_cut} we obtain that the $F$-cut locus of $x$ is
\begin{equation*}
\begin{split}
\mathcal{C}^F_x&=\varphi(d_h(x,\hat{q}_0)=\varphi(\pi R,\hat{q}_0)\\&=(2a-r(x),\pi(1+\mu R),
\end{split}
\end{equation*}
where $R$ is radius of round sphere.

\item If the $F$-cut locus of $x\in M\setminus\{p,q\}$ is a single point, say $q\in M$, then $\hat{q}:=\varphi(-l,q)$ is a $h$-cut point, where $d_F(x,q)=l$. Obviously $\hat{q}$ is the only $h$-cut point of $h$ due to the Proposition \ref{prop_F_cut}.

Since the $h$-cut locus of $x\in M\{p,q\}$ is made of a single point $\hat{q}$, we know from \cite{ST} that $G=\frac{1}{R^2}$ must be a positive constant and hence $(M,h)$ is actually the round sphere of radius $R$.

Taking now into account that $(M,h)$ is a constant Gaussian curvature Riemannian surface and $W$ a Killing field on $(M,h)$, it follows from \cite{BRS} that the corresponding Randers metric by the Zermelo navigation must be of constant flag curvature.

$\qedd$
\end{enumerate}

\begin{remark}
We recall that in order to obtain all $F$-geodesics $\mathcal{P}(s)$ with $\frac{d\mathcal{P}}{ds}>0$ emanating from a point $p=\mathcal{P}(0)$, $\theta(p)=0$, we need to consider the Riemannian geodesics $\gamma^p_\nu$ with Clairaut constant $\nu\in(-\mu\cdot m^2(a),m(a))$.

On the other hand, we have determined the $F$-cut point $q$ of a point $p=\mathcal{P}(0)$ by mapping the $h$-cut point $\hat{q}$ of the same point $p=\gamma^p_\nu(0)$ along the corresponding $h$-geodesic $\gamma^p_\nu=\varphi(-s,\mathcal{P}(s))$, and using the structure of the $h$-cut locus for such $h$-geodesics determined in \cite{ST}, for $\nu\in(0,m(a))$. Hence, strictly speaking we have determined only the $F$-cut locus of a point $p$ along the $F$-geodesics corresponding to the $h$-geodesics having $\nu\in(0,m(a))$, having out the $h$-geodesics corresponding to $\nu\in(-\mu\cdot m^2(a),0)$.

However, in the Riemannian case, due to the reversibility of $h$-geodesics and symmetry of the distance function $d_h$, it is easy to observe that the cut locus of a point $p\in M$ made of $h$-cut points along the $h$-geodesics $\gamma^p_\nu$, $\nu\in(-\mu\cdot m^2(a),0)$ is actually a subset of the $h$-cut locus of $p$ made of $h$-cut points along all $h$-geodesics $\gamma^p_\nu$, $\nu\in(0,m(a))$, hence we are not missing any points in $\mathcal{C}^F_p$.
\end{remark}

\section{The behaviour of cut locus when the cut locus of a point on equator is the subarc of the equator}\label{proof_thm_F_cut_2}
\subsection{The cut locus}

\quad From the previous section, we can see that, if the Gaussian curvature is monotone non-decreasing (increasing) then $h$-half period function is monotone non-increasing (decreasing), but the inverse is not true, i.e. if $h$-half period function is monotone non-increasing does not implies Gaussian curvature is monotone non-decreasing.

In this section, we will consider the more general case by extending the results in \cite{BCST} to the Randers case.

Let $(M,h)$ be the Riemannian 2-sphere of revolution considered in the previous sections, but in this section we 
do not assume the second condition in Remark \ref{rem_2_sphere_properties},
and let $W=\mu\cdot\frac{\partial}{\partial\theta}$ the wind blowing along the parallels, $\mu<\left\{\frac{1}{\max m(r)}:r\in[0,2a]\right\}$. If we denote by $(M,F=\alpha+\beta)$ the Randers rotational constructed from navigation data $(h,W)$ in Section 2.2, then we have
\begin{proof}[Proof of Theorem \ref{thm_F_cut_2}]
If the cut locus of a point $q$ on $\{r=a\}$ is a subarc of $\{r=a\}$, since the equator is invariant under the flow action, then by Proposition \ref{prop_F_cut} it follows that the $h$-cut locus of the point $q$ is a subarc of $\{r=a\}$. Hence, by using Theorem 3.5 in \cite{BCST} it results that the $h$-cut locus of the point $\widetilde{q}$ is a subarc of the antipodal parallel $\{r=2a-r(\widetilde{q})\}$.

Taking now into account that any parallel is flow-invariant by Proposition \ref{prop_F_cut} it follows that the $F$-cut locus of $\widetilde{q}$ must be a subarc in the antipodal parallel $\{r=2a-r(\widetilde{q})\}$. Clearly, the $F$-cut locus is obtained by rotating the $h$-cut locus by flow action on the parallel $\{r=2a-r(\widetilde{q})\}$.

$\qedd$
\end{proof}

\subsection{Examples}\label{ex_F_cut}
\begin{example}\label{ex_1}
Let us consider the Riemannian 2-sphere of revolution $M_\lambda:=(\mathbb{S}^2,h_\lambda)$, introduced in \cite{BCST}, where
\begin{equation}\label{h_lambda}
h_\lambda=dr^2+m^2_\lambda(r)d\theta^2
\end{equation}
and
\begin{equation}\label{ex1_m}
m_\lambda(r)=\frac{\sqrt{\lambda+1}\cdot\sin r}{\sqrt{1+\lambda\cos^2 r}},\quad \lambda\geq 0.
\end{equation}
It is clear that the function $r\mapsto m_\lambda(r)$ is symmetric with respect to the equator $\{r=\frac{\pi}{2}\}$, and a straightforward computation shows that the Gaussian curvature of $(\mathbb{S}^2,h_\lambda)$ is
\begin{equation*}
G_\lambda(r)=\frac{(\lambda+1)(1-2\lambda\cos^2 r)}{(1+\lambda\cos^2 r)^2}.
\end{equation*}
For $\lambda=0$ one obtains the the round sphere $\Sph^2$ with canonical Riemannian metric and for $\lambda\to\infty$ the metric
\begin{equation*}
h_\infty=dr^2+\tan^2 r d\theta^2,
\end{equation*}
that is singular along the equator $\{r=\frac{\pi}{2}\}$.

By taking the derivative of $G_\lambda$ one can see that $G_\lambda$ is not monotone along the meridian from a pole to the equator. Indeed, we have
\begin{equation*}
G'_\lambda(r)=\frac{2\lambda(\lambda+1)\sin 2r}{(1+\lambda\cos^2r)^3}(2-\lambda\cos^2r).
\end{equation*}

On the other hand, more computations lead to
\begin{equation*}
\mathcal{H}(\nu)=\pi-\frac{\lambda\pi\nu}{\sqrt{\lambda+1}\sqrt{\left(\lambda+1+\lambda\nu^2\right)}},\quad \lambda>0,\quad  \nu\in(0,\sqrt{\lambda+1}),
\end{equation*}
where we use $\xi(\nu)=\nu^2$, and from here
\begin{equation*}
\mathcal{H}'(\nu)=\frac{-\pi\lambda\sqrt{\lambda+1}}{\left(\lambda+1+\lambda\nu^2\right)^\frac{3}{2}},\quad \lambda>0,
\end{equation*}
moreover
\begin{equation*}
\mathcal{H}''(\nu)=\frac{3\pi\lambda^2\nu\sqrt{\lambda+1}}{\left(\lambda+1+\lambda\nu^2\right)^\frac{5}{2}},\quad \lambda>0,
\end{equation*}
see \cite{BCST} for detailed computations.

Then Lemma 3.3 in \cite{BCST} implies that the $h$-cut locus of a point $q$ on $\{r=\frac{\pi}{2}\}$ is a subarc in $\{r=\frac{\pi}{2}\}$ and hence by Theorem 3.5 in \cite{BCST} it results that for this 2-sphere of revolution, the cut locus of any point $\widetilde{q}\in\Sph^2$, $r(\widetilde{q})\in(0,\pi)\setminus\{\frac{\pi}{2}\}$ is a subarc of the antipodal parallel $\{r=2a-r(\widetilde{q})\}$ (see \cite{BCST} for details).

\bigskip

Let us consider the associated Randers rotational metric $F=\alpha+\beta$ obtained by Zermelo's navigation method   (\cite{HCS}, \cite{R}) from the navigation data $(h_\lambda,W)$, where $W=\mu\cdot\frac{\partial}{\partial\theta}$, $\mu<\left\{\frac{1}{\max m_\lambda(r)}:r\in[0,\pi]\right\}=\frac{1}{m_\lambda(\frac{\pi}{2})}=\frac{1}{\sqrt{\lambda+1}}$. From Proposition \ref{Prop_Randers_metric} it follows
\begin{equation*}
(a_{ij})=\left(
\begin{array}{cc}
\frac{1+\lambda\cos^2 r}{1+\lambda\cos^2 r-\mu^2(\lambda+1)\sin^2r} & 0\\
0 & \frac{((\lambda+1)\sin^2r)(1+\lambda\cos^2 r)}{(1+\lambda\cos^2 r-\mu^2(\lambda+1)\sin^2r)^2}
\end{array}\right)
\text{, }
b_i=\left(
\begin{array}{c}
0 \\
\frac{-\mu(\lambda+1)\sin^2r}{1+\lambda\cos^2 r-\mu^2(\lambda+1)\sin^2r}
\end{array}\right).
\end{equation*}
For the sake of simplicity, let us consider
\begin{equation*}
\mu=\frac{1}{2}\cdot\frac{1}{\sqrt{\lambda+1}}.
\end{equation*}
Then \eqref{F_Half_front} implies
\begin{equation*}
\mathcal{H}^+_F(\nu)=\pi-\frac{\lambda\pi\nu}{\sqrt{\lambda+1}\sqrt{\lambda+1+\lambda\nu^2}}+\frac{1}{\sqrt{\lambda+1}}\left(\frac{\pi}{2}-\nu^2\right),\quad \lambda>0
\end{equation*}
and therefore
\begin{equation*}
\begin{split}
(\mathcal{H}^+_F)'(\nu)=\frac{-\lambda\pi\sqrt{\lambda+1}}{(\lambda+1+\lambda\nu^2)^\frac{3}{2}}-\frac{2\nu}{\sqrt{\lambda+1}}, \quad \lambda>0,
\end{split}
\end{equation*}
and
\begin{equation}\label{ex1_H''_F}
(\mathcal{H}^+_F)''(\nu)=\frac{3\pi\lambda^2\nu\sqrt{\lambda+1}}{(\lambda+1+\lambda\nu^2)^\frac{5}{2}}-\frac{2}{\sqrt{\lambda+1}},\quad \lambda>0.
\end{equation}
We observe that if $\mathcal{H}(\nu)$ is monotone non-increasing, then $\mathcal{H}^+_F(\nu)$ is decreasing on $\nu\in(0,\sqrt{\lambda+1})$.

Moreover, observe that the $F$-cut locus of any point $q$ in $\{r=\frac{\pi}{2}\}$ is a subarc of $\{r=\frac{\pi}{2}\}$, as well as, that the $F$-cut locus of any point $\widetilde{q}\in M_\lambda$, such that $r(\widetilde{q})\in(0,\pi)\setminus\{\frac{\pi}{2}\}$ is a subarc of the antipodal parallel $\{r=\pi-r(\widetilde{q})\}$. Indeed, taking into account the $h$-cut locus of the points $q$ and $\widetilde{q}$, respectively and the fact that the equator and parallels are invariant under the flow, the $F$-cut locus can be obtained from Proposition \ref{prop_F_cut}.

Therefore, we obtain
\begin{proposition}\label{prop_ex_1}
Let $(\Sph^2,F_\lambda=\alpha+\beta)$ be the Randers rotational metric induced from the navigation data $(h_\lambda,W)$ on $\Sph^2$ given by \eqref{h_lambda}, \eqref{ex1_m}. If $\lambda>0$, then
\begin{itemize}
\item[(i)] the cut locus of a point $q\in\Sph^2$ on the equator is a subarc of the equator. 
\item[(ii)] the cut locus of a point $\widetilde{q}\in\Sph^2$, distinct from the pair of poles, is a subarc of the antipodal parallel $\{r=\pi-r(\widetilde{q})\}$.
\end{itemize} 
\end{proposition}

This is the generalization of the first part of Theorem 4.4 in \cite{BCST} to the Randers case. Observe that due to Lemma \ref{F_curvature} and the formula for $G'_\lambda$ it follows that the Randers rotational metric constructed in this example is not of monotone flag curvature along meridian.
 
\begin{remark}
The Riemannian 2-sphere of revolution $(\Sph^2,h_\lambda)$ given by \eqref{h_lambda}, \eqref{ex1_m}, $\lambda\geq 0$ gives an example for Theorem 3.6 in \cite{BCST} due to the fact that the $h$-half period function satisfies
\begin{equation*}
\mathcal{H}'(\nu)<0<\mathcal{H}''(\nu) \text{ for any } \lambda>0.
\end{equation*}
However, this type of relation is not true in the Randers case. Indeed, even though the $h$- and $F$-half period function $\mathcal{H}(\nu)$ and $\mathcal{H}^+_F(\nu)$ have the same monotonicity, respectively, they do not share the same convexity. Formula \eqref{ex1_H''_F} implies that 
$(\mathcal{H}^+_F)''(\nu)$ is not always positive. For instance, numerical simulations show that $(\mathcal{H}^+_F)''(\nu)\leq 0$, for $\lambda \leq 1.5$, while for $\lambda> 1.5$ the function $(\mathcal{H}^+_F)''(\nu)$ can take both, positive and negative values,  where $\nu\in(0,\sqrt{\lambda+1})$, see Figure \ref{fig_ex1}.
\end{remark}
\begin{figure}[H]
  \centering
   \includegraphics[width=6cm,height=5cm]{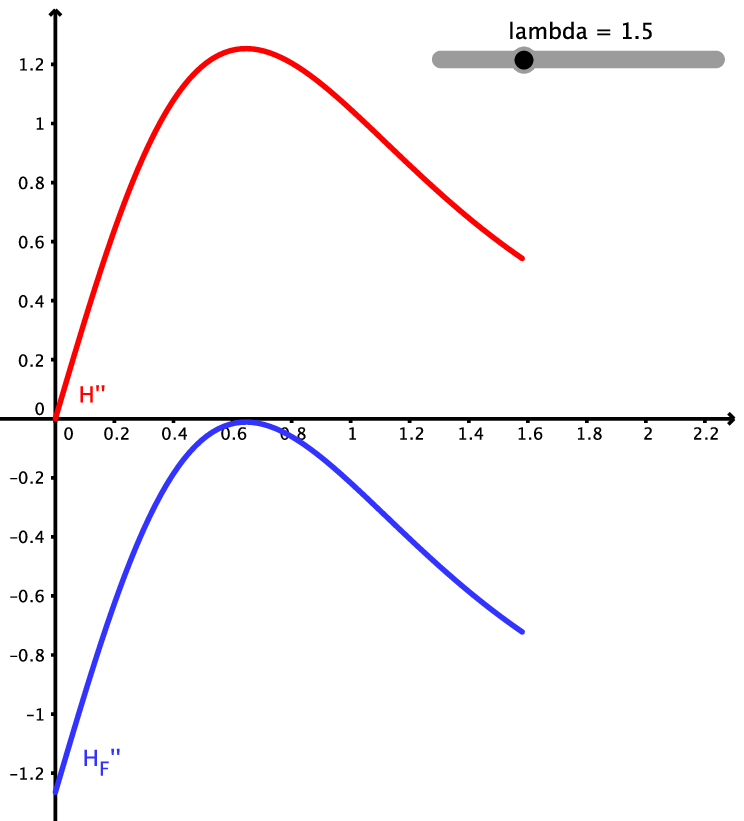}
   \quad
    \includegraphics[width=6cm,height=5cm]{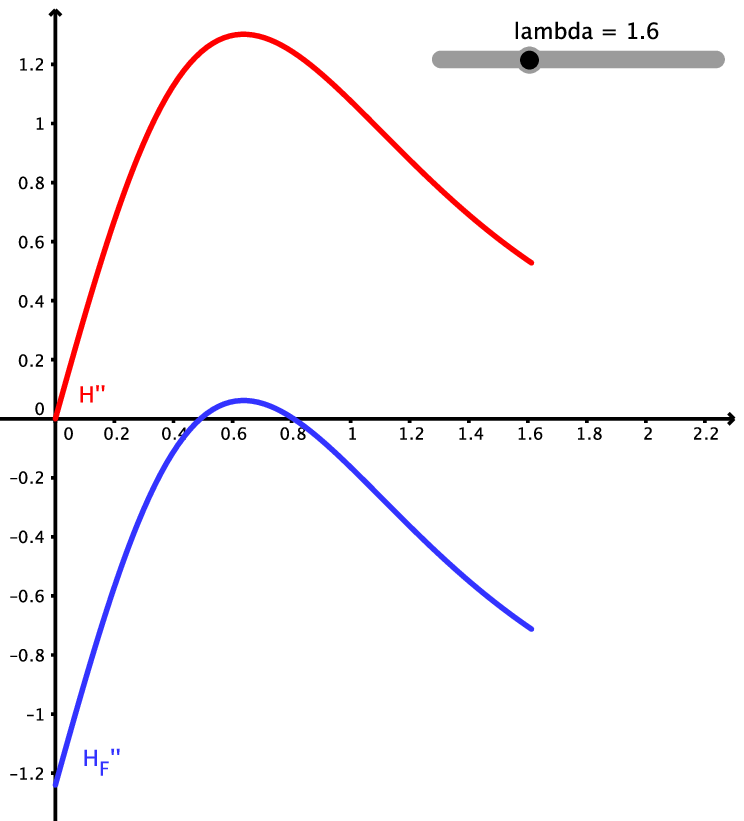}
      \caption{The graphs of $\mathcal{H}''(\nu)$ and $(\mathcal{H}^+_F)''(\nu)$, where $\lambda =1.5$ and $\lambda=1.6$ respectively.}
      \label{fig_ex1}
\end{figure}
\end{example}
\begin{example}
Another example is obtained from the Riemannian 2-sphere of revolution $(\Sph^2,h_\lambda)$ given in \cite{BCJ}, where $h_\lambda$ is given by \eqref{h_lambda} and
\begin{equation*}
m_\lambda(r)=\frac{\sin r}{\sqrt{1-\lambda\sin^2r}},\quad r\in[0,\pi],\quad \lambda\in(0,1).
\end{equation*}
By straightforward computation one can see that
\begin{equation*}
G_\lambda(r)=\frac{(1-\lambda)-2\lambda\cos^2r}{(1-\lambda\sin^2r)^2}
\end{equation*}
and
\begin{equation*}
G'_\lambda(r)=\frac{4\lambda\sin r\cos r (2(1-\lambda)-\lambda\cos^2r)}{(1-\lambda\sin^2 r)^3}.
\end{equation*}
It is clear that for $\lambda\in(0,1)$, $G'_\lambda$ vanishes at the pair of poles and the equator and the Gaussian curvature, $G_\lambda$ is monotone for $\lambda\in(0,\frac{2}{3})$ with a local extremum of $\lambda=\frac{2}{3}$ (see \cite{BCJ} for details).

A similar computation with \cite{BCST} shows that
\begin{equation*}
\mathcal{H}(\nu)=\pi-\frac{\pi\nu\lambda}{\sqrt{1+\lambda\nu^2}},\quad \nu\in\left(0,\sqrt{1-\lambda}\right)
\end{equation*}
and hence
\begin{equation*}
\mathcal{H}'(\nu)=\frac{-\pi\lambda}{(1+\lambda\nu^2)^{\frac{3}{2}}},\quad \mathcal{H}''(\nu)=\frac{3\pi\lambda^2\nu}{(1+\lambda\nu^2)^{\frac{5}{2}}}.
\end{equation*}
(compare with the form in \cite{BCJ} obtained in Hamiltonian formalism).

One can easily see that
\begin{equation}\label{ex_2}
\mathcal{H}'(\nu)<0<\mathcal{H}''(\nu)
\end{equation}
and hence the $h$-cut locus of a point $q$ on the equator is a subarc of the equator (Lemma 3.3 in \cite{BCST}), and the $h$-cut locus of a point $\widetilde{q}$, distinct from equator of $(\Sph^2,h_\lambda)$ is a subarc of the opposite parallel (Lemma 3.4 in \cite{BCST}).

\bigskip

If we consider again the Randers rotational metric $(\Sph^2,F_\lambda=\alpha+\beta)$ obtained by Zermelo's navigation method (\cite{HCS}, \cite{R}) from navigation data $(h_\lambda,W)$, $W=\mu\cdot\frac{\partial}{\partial\theta}$, $\mu<\sqrt{1-\lambda}$, then \eqref{F_Half_front} gives
\begin{equation*}
\mathcal{H}^+_F(\nu)=\pi-\frac{\pi\nu\lambda}{\sqrt{1+\lambda\nu^2}}+\sqrt{1-\lambda}\left(\frac{\pi}{2}-\nu^2\right),\quad \nu\in(0,\sqrt{1-\lambda}),
\end{equation*}
where we consider for simplicity $\mu =\frac{1}{2}\sqrt{1-\lambda}$, and hence
\begin{equation*}
\begin{split}
(\mathcal{H}^+_F)'(\nu)&=\frac{-\pi\lambda}{(1+\lambda\nu^2)^{\frac{3}{2}}}-2(\sqrt{1-\lambda})\nu,\\
(\mathcal{H}^+_F)''(\nu)&=\frac{3\pi\lambda^2\nu}{(1+\lambda\nu^2)^{\frac{5}{2}}}-2\sqrt{1-\lambda}.
\end{split}
\end{equation*}
By a similar argument with Example \ref{ex_1} it follows that Proposition \ref{prop_ex_1} is true for this example as well.
\begin{remark}
When we consider the second derivative of the $F$-half period function $\mathcal{H}_F(\nu)$, we observe that, even through the Riemannian counter part satisfies \eqref{ex_2}, in the Finsler case we have
$(\mathcal{H}^+_F)'(\nu)<0 $, however 
 $(\mathcal{H}^+_F)''(\nu)\leq 0$, for $\lambda \leq 0.6$, while for $\lambda> 0.6$ the function $(\mathcal{H}^+_F)''(\nu)$ can take both, positive and negative values,  where $\nu\in(0,\sqrt{1-\lambda})$, see Figure \ref{fig_ex2}.
that is, in this case also the convexity of the half period function in the Riemannian and Finsler case are quite different.
\end{remark}

\begin{figure}[H]
  \centering
    \includegraphics[width=6cm,height=5cm]{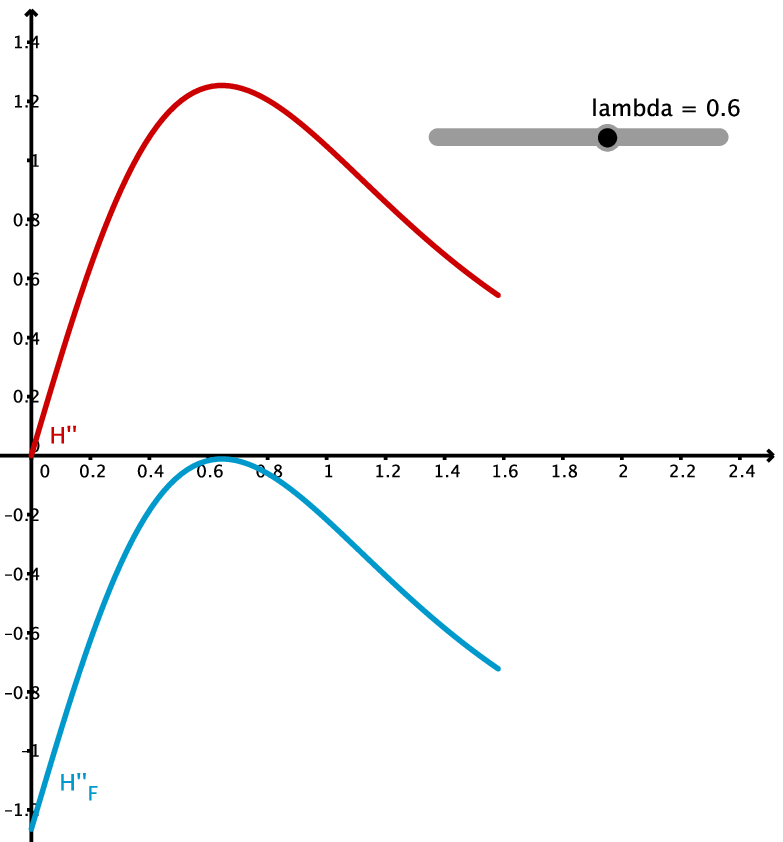}
    \quad
        \includegraphics[width=6cm,height=5cm]{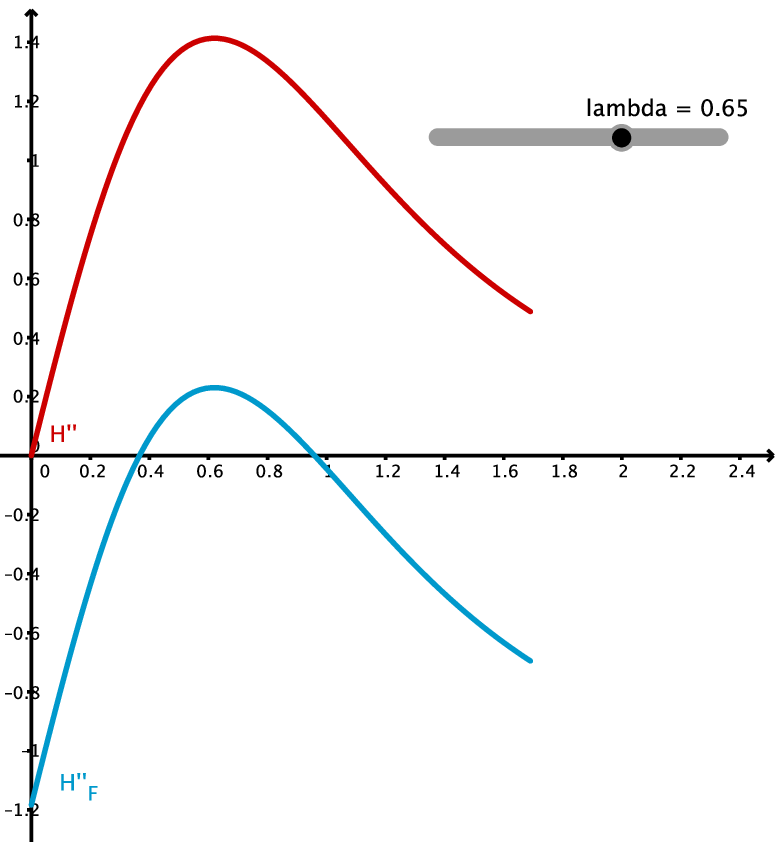}
      \caption{The graphs of $\mathcal{H}''(\nu)$ and $(\mathcal{H}^+_F)''(\nu)$, where $\lambda=0.6$ and $\lambda =0.65$ respectively.}
      \label{fig_ex2}
\end{figure}

\end{example}


\bigskip

\noindent
 Rattanasak HAMA\\
Department of Mathematics,\\
Faculty of Science, KMITL, Bangkok, 10520, Thailand

\medskip
{\tt rattanasakhama@gmail.com}

\medskip

\noindent
Jaipong KASEMSUWAN\\
Department of Mathematics,\\
Faculty of Science, KMITL, Bangkok, 10520, Thailand

\medskip
{\tt jaipui@hotmail.com
}

\medskip

\noindent 
Sorin V. SABAU\\
School of Biological Sciences,
Department of Biology,\\
Tokai University,
Sapporo 005\,--\,8600,
Japan

\medskip
{\tt sorin@tokai.ac.jp}


\begin{thebibliography}{M}

\bibitem
{AT}
     M. Abate, F. Tovena, 
Curves and surfaces, Springer,
2012.

\bibitem
{BCS}
     D.~Bao, S. S.~Chern, Z.~Shen,
An Introduction to Riemann Finsler Geometry, Springer, GTM
\textbf{200},
2000.

\bibitem{BCJ}
B. Bonnard, J. B. Caillau and G. Janin, {\it Conjugate-cut loci and injectivity domains of two-spheres of revolution}, ESAIM: COCV 19 (2013), 533-554.

\bibitem
{BCST}
B. Bonnard, J. B. Caillau, R. Sinclair, M. Tanaka, {\it Conjugate and cut loci of a two-sphere of revolution with application to optimal control}, Ann. I. H. Poincare-AN 26 (2009), 1081--1098.

\bibitem
{BR}
     D.~Bao, C. Robles, 
{\it Ricci and flag curvatures in Finsler geometry},
      in A Sampler of Riemann-Finsler Geometry, MSRI Series 50, 2004, 197--259.

\bibitem
{BRS}
     D.~Bao, C. Robles, Z. Shen 
{\it Zermelo Navigation on Riemannian manifolds},
      J. Diff. Geom, 66(2004), 377--435. 

\bibitem
{HCS}
R. Hama, P. Chitsakul, S. V. Sabau,{\it The Geometry of a Randers rotational surface}, 2015, Publ. Math. Debrecen, 87/3-4 (2015), 473--502.

\bibitem
{HK}
R. Hama, J. Kasemsuwan,{\it The Theory of Geodesics on Some Surface of Revolution}, KMITL Science and Technology Journal, 17 (2017) no. 1, 42--47.


\bibitem
{R}
C. Robles, 
{\it Geodesics in Randers spaces of constant curvature},
      Trans. AMS 359 (2007), no. 4, 1633--1651. 

\bibitem
{SST}
K.~Shiohama, T.~Shioya and M.~Tanaka, 
The Geometry of Total Curvature on Complete Open Surfaces, 
Cambridge tracts in mathematics \textbf{159}, 
Cambridge University Press, Cambridge, 2003.

\bibitem{ST}
R. Sinclair, M. Tanaka
{\it The cut locus of a two sphere of revolution and Toponogov's comparison theorem},
 Tokyo Math. J. 59(2007), 379--399.
 
 \bibitem{SaT}
S. V. Sabau, M. Tanaka,
{\it The cut locus and distance function from a closed subset of a Finsler manifold}, Houston Journal of Mathematics, Vol. 42, No. 4 (2016), 1157-1197.

\bibitem{Mat}
M. Matsumoto, Finsler Geometry in the 20th-Century, In: Handbook of Finsler Geometry, (ed. P. L. Antonelli), Kluwer Academic Publishers, 2004.

\bibitem{S}
Z. Shen, {\it Finsler metrics with K=0 and S=0}, Canadian J. Math. 55 (2003), 112-132.
\end{thebibliography}
\end{document}